\documentclass[11pt,a4paper]{amsart}
\usepackage[english]{babel}
\usepackage[T1]{fontenc} 
\usepackage{mathpazo}
\usepackage{upref}
\usepackage{enumerate}
\usepackage{graphicx,color}
\usepackage{dsfont}       
\usepackage[colorlinks=true, pdftitle={j-elliptic functionals},
pdfauthor={Ralph Chill, Daniel Hauer, James Kennedy}]{hyperref}
\usepackage{amsmath,amssymb,amsthm}

\numberwithin{equation}{section}

\theoremstyle{theorem}
\newtheorem{theorem}{Theorem}[section]

\newtheorem{lemma}[theorem]{Lemma}
\newtheorem{corollary}[theorem]{Corollary}

\theoremstyle{definition}
\newtheorem{definition}[theorem]{Definition}

\theoremstyle{remark}
\newtheorem{remark}[theorem]{Remark}

\newcommand\R{{\mathbb{R}}}
\newcommand\N{\mathbb{N}}

\newcommand\dx{\mathrm{d}x }

\newcommand\ds{\mathrm{d}s }

\newcommand\dH{\mathrm{d}\mathcal{H} }

\newcommand\eR{{\mathds{R}}\cup \{ +\infty\}}
\newcommand\V{{\mathrm V}} 
\newcommand\HH{{\mathrm H}}
\newcommand\D{{\mathrm D}}
\newcommand\E{\varphi}
\newcommand\Eh{\varphi^\HH}

\DeclareMathOperator{\T}{tr}

\DeclareMathOperator{\kernel}{ker}
\DeclareMathOperator{\sign}{sign}

\newcommand\abs[1]{\lvert#1\rvert} 
 
\newcommand\norm[1]{\lVert#1\rVert}

\definecolor{darkred}{rgb}{0.7,0.1,0.1}

\newcommand{\vanish}[1]{\relax}

\begin{document}

\title{Nonlinear semigroups generated by $j$-elliptic functionals}

\author{Ralph Chill}
\address{Institut f\"ur Analysis, Fachrichtung Mathematik, TU
Dresden, 01062 Dresden, Germany.}
\email{ralph.chill@tu-dresden.de}

\author{Daniel Hauer}
\address{School of Mathematics and Statistics, The University
of Sydney, NSW 2006, Australia.}
\email{daniel.hauer@sydney.edu.au}

\author{James Kennedy}
\address{Institut f\"ur Analysis, Dynamik und Modellierung, 
Universit\"at Stuttgart, Pfaffenwaldring 57,
70569 Stuttgart, Germany.}
\email{james.kennedy@mathematik.uni-stuttgart.de}



\date{\today}

\begin{abstract}
  We generalise the theory of energy functionals used in the study of
  gradient systems to the case where the domain of definition of the
  functional cannot be embedded into the Hilbert space $H$ on which the
  associated operator acts, such as when $H$ is a trace space. We show
  that under weak conditions on the functional $\E$ and the map $j$ from
  the effective domain of $\E$ to $H$, which in opposition to the
  classical theory does not have to be injective or even continuous, the
  operator on $H$ naturally associated with the pair $(\E,j)$
  nevertheless generates a nonlinear semigroup of contractions on
  $H$. We show that this operator, which we call the $j$-subgradient of
  $\E$, is the (classical) subgradient of another functional on $H$, and
  give an extensive characterisation of this functional in terms of $\E$
  and $j$.  In the case where $H$ is an $L^2$-space, we also
  characterise the positivity, $L^\infty$-contractivity and existence of
  order-preserving extrapolations to $L^q$ of the semigroup in terms of
  $\E$ and $j$.  This theory is illustrated through numerous examples,
  including the $p$-Dirichlet-to-Neumann operator, general Robin-type
  parabolic boundary value problems for the $p$-Laplacian on very rough
  domains, and certain coupled parabolic-elliptic systems.
\end{abstract}

\maketitle

\section{Introduction}

The theory of energy functionals on Hilbert spaces and their
subgradients has prov\-en to be a powerful tool for the study nonlinear
elliptic and parabolic partial differential equations
\cite{Att84,Barbu2010,Br73,Ze90IIb}.  Not only can existence and
uniqueness of solutions be established with minimal effort by
variational principles, the variational approach also allows one to
prove results about the regularity of solutions, maximum or comparison
principles and the large-time behaviour of solutions in the case of
parabolic problems. Very often, this theory is a natural generalisation
to the nonlinear case of the corresponding theory of sesquilinear forms
used in the study of {\em linear} elliptic and parabolic equations
\cite{DaLi88II,DaLi92V,DaLi93VI,Li69}; in that case, the form is defined
on a Hilbert space $\V$ and induces an operator on another Hilbert space
$\HH$. A key point in the whole theory is that the space $\V$ is
canonically embedded in $\HH$, that is, that there exists a bounded
injection $i:\V\to \HH$. One can however find a plethora of examples
which do not fit into this framework, although one would expect (or
hope) that variational methods should still be applicable; as a
prototype consider the case where $\HH$ is a trace space of $\V$.

Recently, Arendt and ter Elst developed a general theory of
\emph{$j$-elliptic forms} \cite{ArEl11,ArEl12}, see also \cite{AEKS14},
where the embedding $i$ is replaced with a closed linear map
$j:\V\supseteq D(j)\to\HH$, which is however not necessarily injective
or even bounded. This allowed them to develop a rich variational theory
of the \emph{Dirichlet-to-Neumann operator} acting on functions defined
on the boundary $\partial\Omega$ of a general (bounded) open set $\Omega
\subseteq \R^d$. We recall that the Dirichlet-to-Neumann operator
assigns to each boundary function $g\in H:=L^2(\partial\Omega)$
(Dirichlet data) the outer normal derivative $\frac{\partial
  u}{\partial\nu} \in L^2(\partial\Omega)$ (Neumann data) of the
solution $u \in V:= H^1(\Omega)$ of the Dirichlet problem
\begin{displaymath}
\begin{aligned}
	\Delta u &=0 &\qquad &\text{in $\Omega$}\\
	u &=g & &\text{on $\partial\Omega$},
\end{aligned}
\end{displaymath}
if such a function $\frac{\partial u}{\partial\nu}$ exists in
$L^2(\partial\Omega)$.

A corresponding variational theory of $p$-Dirichlet-to-Neumann operators
on Lip\-schitz domains, via energy functionals, analogous to the
theory of Arendt and ter Elst, was recently developed by one of the
current authors \cite{Ha15}; to the best of our knowledge this was the
first systematic treatment of this family of operators.

In this paper, we shall construct a general theory of \emph{$j$-elliptic
  energy functionals}, which will allow us to incorporate and treat
$p$-Dirichlet-to-Neumann operators together with various other types of
operators, including the $p$-Laplacian with Robin boundary
conditions on rough domains, and certain coupled parabolic-elliptic
systems, within the one unified framework. Along the way, we shall show
that many known, or ``classical'', results from the theory of energy
functionals and nonlinear semigroups on a Hilbert space can be
readily adapted to this setting.

In Section \ref{sec.subgradient} we lay the foundations of our abstract
theory. On a given Hilbert space $H$, given an energy functional $\E$ on 
$V$, we introduce the natural, possibly
multivalued operator associated with the pair $(\E, j)$, which we call
the \emph{$j$-subgradient} $\partial_j\E$ of $\E$. Then under natural
assumptions on $\E$ including a $j$-analogue of coercivity or convexity,
as well as the assumption that $j$ is weak-to-weak continuous, we show
that $\partial_j\E$ is cyclically monotone and even maximal monotone;
see Lemma~\ref{thm:gradient-cyclic} and
Theorem~\ref{thm:max-monoton-partial-j-E}. This allows us to derive a
parabolic generation result by invoking known results from the
literature; see Theorem~\ref{thm:generation} in Section
\ref{sec.invariance}. More precisely, the negative $j$-subgradient
$-\partial_j\E$ generates a strongly continous semigroup
$S=(S(t))_{t\geq 0}$ of nonlinear Lipschitz continuous mappings $S(t)$
on the closure of $j(D(\E))$ in $H$. We shall use
the notation $S\sim (\E,j)$ to say that the semigroup $S$ is
generated by the negative $j$-subgradient $-\partial_j\E$.

The observation that $\partial_j\E$ is cyclically monotone also implies
the existence of a ``classical'' functional $\E^H$ on the Hilbert space
$H$ such that $\partial_j \E \subseteq \partial \E^H$, analogous to the
corresponding statement for forms. We give an extensive characterisation
of the functional $\E^H$ in Section \ref{sec.identification}; see
Theorem~\ref{thm.eh.identification}. The case when $j$ is merely a
closed linear map, not necessarily weak-to-weak continuous, is treated
in Section \ref{sec.j.weakly.closed}; importantly, it turns out that this
seemingly more general case can be reduced to the one considered 
earlier via a simple argument introducing a new, related space and map.

Section \ref{sec.invariance} is devoted to an important extension of the
existing theory, namely the characterisation of invariance principles
of closed convex sets under the action of a semigroup $S\sim (\E,j)$
under the assumption that $\E$ is convex, proper
and $j$-elliptic, see Theorem~\ref{thm:beurling}. As in
the classical case, this allows us to give characterising conditions on
the associated functionals under which two semigroups can be compared
(Theorem~\ref{thm:comparison-property}), which as a special case leads
to order-preserving and dominating semigroups
(Corollaries~\ref{cor:order-preservingness} and
\ref{cor:domination-property}, respectively) if our Hilbert space $H$ is
of the form $L^2(\Sigma)$ for a measure space $\Sigma$.
Continuing with the $L^2$ theme, we give a characterisation under which
the semigroup in question, assumed to be order preserving on
$L^2(\Sigma)$, can be extrapolated to an order-preserving semigroup on
the whole scale of $L^q(\Sigma)$-spaces and even the whole scale of 
Orlicz $L^\psi$-spaces; see Theorem~\ref{thm:L1-contractivity}.

In Section \ref{sec.examples}, we illustrate our theory through four
different examples. In Section \ref{sec.robin}, we introduce a weak
formulation of a nonlinear parabolic problem on a domain $\Omega$ with
Robin boundary conditions and very weak assumptions on the boundary
$\partial \Omega$, showing that there is a strongly continuous semigroup
on $L^2(\Omega)$ solving this problem; see
Theorem~\ref{thm:robin-semigroup}. This puts results from Daners and Dr\'abek \cite{DaDr09}
into a general framework; see also Warma \cite{Wa14} who considered non-local 
Robin boundary conditions.

In Section \ref{sec.Dirichlet-to-Neumann}, we consider
$p$-Dirichlet-to-Neumann operators on non-smooth domains; this example
could be thought of as the motivating example for the whole theory (at
least in the smooth case). Here our treatment is relatively brief, the
prime purpose again being the establishment of generation and
extrapolation theorems, see Theorem~\ref{propo:order-preserving}; a more
complete treatment for the non-smooth case as in \cite{ArEl11} will be 
deferred to a later work.

Our next example, in Section \ref{sec.elliptic-parabolic}, is a system
of coupled parabolic-elliptic equations, equivalent to a degenerate
parabolic equation, where, roughly speaking, one takes open sets $\Omega
\subseteq \hat\Omega$ in $\R^d$, solves the usual Cauchy problem
$\partial_t u - \Delta_p u = f$ in $(0,T)\times\Omega$, and demands that
the solution $u$ have an extension $\hat u$ to $\hat\Omega$ which is
$p$-harmonic (i.e. $\Delta_p \hat u=0$) in $\hat\Omega \setminus \Omega$ 
and vanishes on $\partial\hat\Omega$ for each $t \geq 0$. Here, we
denote by $\Delta_{p}$ the celebrated $p$-Laplace operator given by
$\Delta_{p}u=\textrm{div} (\abs{\nabla u}^{p-2}\nabla u)$.

In addition to the well-posedness of the problem, we show that the generated 
semigroup dominates the semigroup generated by the $p$-Laplacian on 
$\Omega$ with Dirichlet boundary conditions.

The final example, in Section \ref{sec.elliptic-parabolic.1}, is a
partial repetition of the example from Section
\ref{sec.elliptic-parabolic}, but for the case $p=1$. The $1$-Laplace
operator serves as an illustration why we write our general theory for
functionals on locally convex topological vector spaces (instead of
Banach spaces).  In this final example we only prove well-posedness of
the underlying coupled parabolic-elliptic system.

\section{The $j$-subgradient and basic properties}
\label{sec.subgradient}

\subsection{Definition and characterisation as a classical gradient}
\label{sec.def.j-subgradient}

Throughout, let $\V$ be a real locally convex topological vector space and
$\HH$ a real Hilbert space equipped with inner product $\langle\cdot , \cdot
\rangle_\HH$ and associated norm denoted by $\|\cdot \|_\HH$. Further,
let $j: \V \to \HH$ be a linear operator which is weak-to-weak
continuous, and denote by $\eR$ the one-sided extended real line.

Given a functional $\E : \V \to \eR$, we call the set $\D (\E) := \{ \E
<+\infty \}$ its {\em effective domain}, and we say that $\E$ is
{\em proper} if the effective domain is non-empty. Its
{\em $j$-subgradient} is the operator
\begin{displaymath}
  \partial_j\E := \Bigg\{ (u,f)\in\HH\times\HH\;\Bigg\vert\;
  \begin{array}[c]{c}
    \exists \hat{u}\in\D (\E )
    \text{ s.t. } j(\hat{u}) = u \text{ and for every } \hat{v}\in\V \\[0,1cm]
    \liminf_{t\searrow 0} \frac{\E (\hat{u} +
      t\hat{v}) -\E (\hat{u})}{t} \geq \langle f,j(\hat{v})\rangle_\HH
  \end{array}
  \Bigg\}.
\end{displaymath}
We shall usually view operators on $\HH$ as relations
$A\subseteq\HH\times\HH$, but we shall also use the notation
\begin{displaymath}
  A(u) := \Big\{ f\in\HH\; \big\vert\; (u,f)\in A\Big\} ,
\end{displaymath}
which suggests that $A$ is a mapping from $\HH$ into $2^\HH$, the power
set of $\HH$, that is, $A$ is a so-called multivalued operator. We take the 
usual definition of the {\em domain} of an operator 
$A\subseteq\HH\times\HH$ as the set
\begin{displaymath}
  D(A) := \Big\{ u\in\HH\; \big\vert\; \exists f\in\HH \text{ s.t. }
  (u,f)\in A \Big\} ,
\end{displaymath}
and similarly for the {\em range} of $A$. We say that
the functional $\E$ is {\em $j$-semiconvex}  if there exists
$\omega\in\R$ such that the ``shifted'' functional 
\begin{align*}
  \E_\omega : \V & \to \eR , \\
  \hat{u} & \mapsto \E (\hat{u}) + \frac{\omega}{2} \, \| j(\hat{u}
  )\|_\HH^2
\end{align*}
is convex, and we say that the functional $\E$ is {\em $j$-elliptic} if
there exists $\omega\geq 0$ such that $\E_\omega$ is convex and
coercive. Saying that a functional $\E$ defined on a locally convex
topological vector space is {\em coercive} means that sublevels $\{ \E
\leq c\}$ are relatively weakly compact for every $c\in\R$. Finally, we
say that the functional $\E$ is {\em lower semicontinuous} if the
sublevels $\{ \E \leq c\}$ are closed in the topology of $\V$ for every
$c\in\R$.

\begin{remark} 
  \label{rem.banach.space.case}
  In the important special case when $\V$ is a Banach space, $j$ is
  weak-to-weak continuous if and only if $j$ is continuous. Moreover, in
  this case, the ``shifted'' functional $\E_\omega$ is lower
  semicontinuous if and only if $\E$ itself is lower
  semicontinuous. Finally, if $\V$ is a {\em reflexive} Banach space,
  then $\E$ is coercive if and only if the sublevels $\{\E \leq c\}$ are
  (norm-) bounded.
\end{remark}

\begin{lemma} 
  \label{lem.easy.ident}
  Let $\V$, $\HH$, $j$ and $\E$ be as above.
  \begin{itemize}
  \item[(a)] If $\E_\omega$ is convex for some $\omega\in \R$, then
    \begin{displaymath}
      \partial_j\E  = \Bigg\{ (u,f)\in\HH\times\HH \;\Bigg\vert\;
      \begin{array}[c]{c}
        \exists \hat{u}\in\D (\E )
        \text{ s.t. }
        j(\hat{u}) = u \text{ and for every } \hat{v}\in\V \\[0,1cm]
        \E_\omega (\hat{u} + \hat{v}) - \E_\omega
        (\hat{u}) \geq \langle f + \omega j(\hat{u}) , j(\hat{v}) \rangle_\HH
      \end{array}
      \Bigg\} .
    \end{displaymath}
  \item[(b)] If $\E$ 
    is G\^ateaux differentiable with G\^ateaux derivative $\E'$, then
    \begin{displaymath}
      \partial_j\E  = \Bigg\{ (u,f)\in\HH\times\HH \;\Bigg\vert\;
      \begin{array}[c]{c}
        \exists \hat{u}\in\D (\E )
        \text{ s.t. } j(\hat{u}) = u \text{ and for every }
        \hat{v}\in\V\\[0,1cm]
        \E'(\hat{u} ) \hat{v}= \langle f,j(\hat{v})\rangle_\HH
      \end{array}
      \Bigg\} .
    \end{displaymath}
  \end{itemize}
\end{lemma}

\begin{proof}
  Let $\omega\in \R$. Then from the limit
  \begin{equation}
    \label{eq:2}
    \lim_{t\searrow 0} \frac{\omega}{2} \, \frac{\| j(\hat{u} +
      t\hat{v}) \|_\HH^2 - \| j(\hat{u})\|_\HH^2}{t}
    = \omega \,  \langle j(\hat{u}) , j(\hat{v}) \rangle_\HH  ,
  \end{equation}
  we obtain first that
  \begin{displaymath}
    \partial_j\E  = \Bigg\{ (u,f)\in\HH\times\HH \;\Bigg\vert\;
    \begin{array}[c]{c}
      \exists \hat{u}\in\D (\E )
      \text{ s.t. } j(\hat{u}) = u \text{ and for every } 
      \hat{v}\in\V \\[0,1cm]
      \liminf_{t\searrow 0}
      \frac{\E_\omega(\hat{u} + t\hat{v}) -\E_\omega (\hat{u})}{t} \geq
      \langle f + \omega j(\hat{u}) ,
      j(\hat{v}) \rangle_\HH
    \end{array}
    \Bigg\} ,
  \end{displaymath}
  which holds for general $\E$. Now claim~(a) follows from the
  assumption that $\E_\omega$ is convex. Claim~(b) is a straightforward
  consequence of the definition of the $j$-subgradient and the G\^ateaux
  differentiability of $\E$.
\end{proof}

\begin{remark} 
  \label{rem.comparison} 
  (a) There exists a well-established classical setting in which
  subgradients of functionals have been defined. This is the setting $\V
  = \HH$ and $j=I$ the identity operator.  The $j$-subgradient then
  coincides with the usual subgradient defined in the literature; see,
  for example, Brezis~\cite{Br73}, Rockafellar~\cite{Rc70}. In this
  classical situation, we call $j$-elliptic functionals simply {\em
    elliptic functionals}, we call the $j$-subgradient simply {\em
    subgradient}, and we write $\partial\E$
  instead of $\partial_j \E$.\\
  (b) Another setting frequently encountered in the literature
  is the case where $\V$ is a Banach space and $j:\V\to\HH$ is a bounded, 
  {\em injective} operator with {\em dense range} (see, for example, J.-L. 
  Lions \cite{Li69}). In other words, $\V$ is a Banach
  space which is continuously and densely embedded into a Hilbert space
  $\HH$. For simplicity, $\V$ may then be identified with a subspace of
  $\HH$ (the range of $j$), so that $j$ reduces to the identity operator
  which is usually neglected in the notation. Identify $\HH$ with its dual space, so that we have a
  Gelfand triple
  \begin{displaymath}
    \V \hookrightarrow \HH = \HH' \hookrightarrow \V' .
  \end{displaymath}

  Let $\E : \V\to\eR$ be a G\^ateaux differentiable functional with
  G\^ateaux derivative $\E' : \V \to \V'$.  By Lemma
  \ref{lem.easy.ident} (b), the $j$-subgradient of $\E$ is then a {\em
    single-valued operator} on the Hilbert space $\HH$ in the sense that
  for every $u\in \HH$ there is at most one $f\in\HH$ such that
  $(u,f)\in \partial_j\E$.  It is then natural to identify
  $\partial_j\E$ with an operator $\HH\supseteq D(\partial_j\E ) \to
  \HH$. By Lemma \ref{lem.easy.ident} (b), this operator coincides with
  the part of the G\^ateaux derivative $\E'$ in $\HH$.  \\
  (c) Conversely, in the setting of (b) above, we may also ``extend''
  the functional $\E$ to the functional $\Eh :\HH\to\eR$ given by
  \begin{displaymath}
    \Eh (u) := 
    \begin{cases}
      \E (\hat{u} ) & \text{if } j(\hat{u}) = u , \\[2mm]
      +\infty & \text{else;}
    \end{cases}
  \end{displaymath}
  this extension is well defined by the injectivity of $j$. A
  straightforward calculation shows that
  \begin{displaymath}
   \partial_j\E = \partial\Eh .
   \end{displaymath}
   Hence, the situation from (b) can be reduced to the situation from
   (a), that is, the situation of classical subgradients. We shall see
   below that this remains true in more general situations.
\end{remark}
\bigskip

We call a finite sequence $((u_i,f_i))_{0\leq i\leq n}$ {\em cyclic} if
$(u_0,f_0)=(u_n,f_n)$. An operator $A \subseteq \HH\times\HH$ is called
{\em cyclically monotone} if for every cyclic sequence
$((u_i,f_i))_{0\leq i\leq n}$ in $A$ one has
\begin{displaymath}
  \sum_{i=1}^{n} \langle f_i , u_{i}-u_{i-1}\rangle_\HH \geq 0 .
\end{displaymath}
Clearly, every cyclically monotone operator is {\em monotone} in the
sense that for every $(u_1,f_1)$, $(u_2,f_2)\in A$ one has
\begin{displaymath}
 \langle f_2 -f_1 , u_2-u_1 \rangle \geq 0 ;
\end{displaymath}
simply choose $n=2$ in the previous inequality.

\begin{lemma} \label{thm:gradient-cyclic} 
  Assume that $\E :\V\to\eR$ is
  convex. Then the $j$-subgradient $\partial_j\E$ is cyclically
  monotone.
\end{lemma}

\begin{proof}
  Let $((u_i,f_i))_{0\leq i\leq n}$ be a cyclic
  sequence in $\partial_j\E$.
  Then there exists a cyclic sequence $(\hat{u}_i)_{0\leq i\leq n}$ in
  $\V$ such that
  \begin{displaymath}
    j(\hat{u}_i ) = u_i \text{ for every } 0\leq i\leq n  ,
  \end{displaymath}
  and, by Lemma \ref{lem.easy.ident} (a), for every $\hat{v}\in\V$ one has
  \begin{align*}
    \E (\hat{u}_1 + \hat{v}) -\E (\hat{u}_1) & \geq
    \langle f_1,j(\hat{v})\rangle_\HH , \\
    \vdots & \phantom{\geq\geq} \quad \vdots \\
    \E (\hat{u}_{n} + \hat{v}) -\E (\hat{u}_{n}) & \geq \langle
    f_{n},j(\hat{v})\rangle_\HH . 
  \end{align*}
  Choosing $\hat{v} = \hat{u}_{i-1} - \hat{u}_i$ in the $i$-th
  inequality, we obtain
  \begin{align*}
    \E (\hat{u}_0 ) -\E (\hat{u}_1)
    & \geq \langle f_1,j(\hat{u}_0)-j(\hat{u}_1) \rangle_\HH , \\
    \vdots & \phantom{\geq\geq} \quad \vdots \\
    \E (\hat{u}_{n-1} ) -\E (\hat{u}_{n}) & \geq \langle
    f_{n},j(\hat{u}_{n-1}) -j(\hat{u}_{n})\rangle_\HH .
  \end{align*}
  Summing the inequalities and using the cyclicity of
  $(\hat{u}_i)_{0\leq i\leq n}$ we obtain
  \begin{displaymath}
    0 \geq \sum_{i=1}^{n} \langle f_i , u_{i-1} - u_i \rangle_\HH ,
  \end{displaymath}
  which implies the claim.
\end{proof}

By \cite[Th\'eor\`eme 2.5, p.38]{Br73}, every cyclically monotone operator
$A$ on a Hilbert space $\HH$ is already contained in a classical
subgradient (that is, $j=I$; see Remark \ref{rem.comparison} (a)
above). More precisely, \cite[Th\'eor\`eme 2.5, p.38]{Br73} and Lemma
\ref{thm:gradient-cyclic} imply the following result.

\begin{corollary}\label{cor:extended-convex-functional}
  Assume that $\E : \V \to\eR$ is convex. Then there exists a convex,
  proper, lower semicontinuous functional $\Eh : \HH \to \eR$ such that
  $\partial_j\E \subseteq \partial\Eh$.
\end{corollary}

We shall identify the functional $\Eh$ under somewhat stronger assumptions 
on $\E$ in Section \ref{sec.identification} below; see Theorem~\ref{thm.eh.identification}.

\begin{theorem}
  \label{thm:max-monoton-partial-j-E}
  Assume that $\E :\V \to\eR$ is convex, proper, lower semicontinuous
  and $j$-elliptic. Then the $j$-subgradient $\partial_j\E$ is maximal
  monotone.
\end{theorem}

\begin{proof}
   Since $\E$ is convex, Lemma~\ref{thm:gradient-cyclic}
  implies that the $j$-subgradient is monotone. By Minty's theorem, it
  suffices to prove that $\omega' I +\partial_j\E$ is surjective for
  some $\omega' >0$. By assumption, we can choose $\omega\geq 0$ such
  that $\E_\omega$ is convex, proper, lower semicontinuous and
  coercive. Now, fix $\omega' >\omega$ and let $f\in\HH$. Then for every
  $\hat{u}\in\V$ and $u:= j(\hat{u})$ we have by definition of the
  $j$-subgradient (or more precisely Lemma~\ref{lem.easy.ident}(a))
  \begin{equation} 
    \label{eq.stationary.problem}
    f\in \omega'\, u + \partial_j\E (u)
  \end{equation}
  if and only if
  \begin{displaymath}
     \E_{\omega'} (\hat{u} +
    \hat{v}) - \E_{\omega'} (\hat{u} )
    \geq \langle f,j(\hat{v}) \rangle_\HH \qquad\text{for all $\hat v\in V$},
  \end{displaymath}
  or
  \begin{displaymath}
     \E_{\omega'} (\hat{u} +
    \hat{v}) - \langle f , j(\hat{u}) + j(\hat{v}) \rangle_\HH \geq
    \E_{\omega'} (\hat{u} ) -
    \langle f,j(\hat{u}) \rangle_\HH \qquad\text{for all $\hat v\in V$.}
  \end{displaymath}
  The latter property is equivalent to
  \begin{displaymath}
    \hat{u} = \arg\min \, \left\{ \E_{\omega'} (\cdot ) -
    \langle f , j(\cdot )\rangle_\HH \right\}.
  \end{displaymath}
  In other words, finding a solution of the stationary problem
  \eqref{eq.stationary.problem} is equivalent to finding a minimiser of
  the functional 
  \begin{math}
  \E_{\omega'} (\cdot ) - \langle f,j(\cdot)\rangle_\HH. 
  \end{math}
  By choice of $\omega$, $\E_\omega$ is convex, lower
  semicontinuous and coercive. Moreover, since $\omega'
  >\omega$, the functional
  \begin{displaymath}
    \V \to \R, \quad \hat{u} \mapsto \frac{\omega' -\omega}{2} 
    \, \| j(\hat{u})\|_\HH^2 -\langle f,j(\hat{u}) \rangle_\HH
  \end{displaymath}
  is convex, lower semicontinuous, and bounded from below. As a
  consequence, $\E_{\omega'} (\cdot )- \langle f , j(\cdot )\rangle_\HH$
  is convex, lower semicontinuous, and coercive. Hence, sublevels of
  this functional are convex, closed, and relatively weakly compact. By
  the Hahn-Banach theorem in the form of \cite[Theorem 3.12, p.66]{Ru73},
  the closure in the topology on $V$ and the weak closure of any convex
  set are identical. Hence sublevels of this functional are weakly
  compact. A standard compactness argument using a decreasing sequence
  of sublevels now implies that the functional above attains its
  minimum, and the claim follows.
\end{proof}

\begin{corollary} 
  \label{cor.extended-quasiconvex-functional} 
  Assume that $\E$ is $j$-semiconvex. Then there exists a
  proper, lower semicontinuous, elliptic functional $\Eh : \HH
  \to\eR$ such that $\partial_j\E \subseteq
  \partial\Eh$. If, in addition, $\E$ is proper, lower semicontinuous 
  and $j$-elliptic, then $\partial_j\E
  = \partial\Eh$, and $\omega\, I + \partial_j\E$ is maximal monotone
  for some $\omega\geq 0$.
\end{corollary}

\vanish{
\begin{proof}
  By assumption, there exists $\omega\geq 0$ such that $\E_\omega$ is
  convex. It follows easily from the definition
  of the $j$-subgradient that
\begin{equation} \label{eq.translated.gradient}
 \omega I + \partial_j\E = \partial_j\E_\omega .
\end{equation}
The claim follows from this identity and Corollary
\ref{cor:extended-convex-functional} applied to $\E_\omega$.
\end{proof}}

\begin{proof}
  By assumption, there exists $\omega\geq 0$ such that $\E_\omega$ is
  convex. Thus by Lemma~\ref{lem.easy.ident} (a) and by definition of
  the $j$-subgradient of $\E_{\omega}$,
  \begin{equation} \label{eq.translated.gradient} 
    \omega I + \partial_j\E = \partial_j\E_\omega .
  \end{equation}
  Since $\E_{\omega}$ is convex, Corollary
  \ref{cor:extended-convex-functional} implies that there is a convex,
  proper and lower semicontinuous functional
  $\E^{H} : H \to \R\cup\{+\infty\}$ such that
  $\partial_j\E_\omega\subseteq \partial\E^{H}$ and so by
  identity~\eqref{eq.translated.gradient},
  $\partial_j\E \subseteq\partial\Eh-\omega I$ holds. Then the
  functional
  $\tilde{\varphi}^{H}:=\varphi^{H}-\tfrac{\omega}{2}\norm{\cdot}_{H}^{2}$
  defined on $H$ is obviously proper, lower semicontinuous and elliptic
  with subgradient
  $\partial\tilde{\varphi}^{H}=\partial\varphi^{H}-\omega I$.
  Hence, replacing $\E^{H}$ with $\tilde{\E}^{H}$ shows that the first
  statement of the corollary holds. Further, the inclusion
  $\partial_j\E_\omega\subseteq \partial\E^{H}$ means that
  $\partial\E^{H}$ is a monotone extension in $H\times H$ of
  $\partial_j\E_\omega$. The additional assumptions that $\E$ is
  proper, lower semicontinuous and $j$-elliptic imply that
  $\partial_{j}\E_{\omega}$ is a maximal monotone set in $H\times H$
  and hence $\partial_j\E_\omega= \partial\E^{H}$. Using again
  identity~\eqref{eq.translated.gradient}, we obtain that
  $\omega\, I + \partial_j\E=\partial\Eh$ is maximal monotone,
  completing the proof of this corollary.
\end{proof}

\subsection{Elliptic extensions} \label{sec.elliptic-extension}

In order to identify the functional $\Eh$ from Corollaries
\ref{cor:extended-convex-functional} and
\ref{cor.extended-quasiconvex-functional}, it is convenient to consider
first the set $\hat{E}_{u}$ of all {\em elliptic extensions} $\hat{u}\in
D(\E)$ of an element $u\in\HH$, which is defined by
\begin{displaymath}
\hat{E}_u = \Big\{ \hat{u}\in D(\E )\;\Big\vert\; j(\hat{u})=u \text{
    and } \liminf_{t\searrow 0} \frac{\E (\hat{u} + t\hat{v}) - \E
    (\hat{u} )}{t} \geq 0 \text{ for every } \hat{v}\in{\rm
    ker}\, j \Big\}.
\end{displaymath}
By using the limit~\eqref{eq:2} and since $\langle j(\hat{u}) ,
j(\hat{v}) \rangle_\HH=0$ for any $\hat{v}\in{\rm ker}\, j$, we see that
\begin{displaymath}
  \hat{E}_u =\Big\{ \hat{u}\in D(\E )\;\Big\vert\; j(\hat{u})=u \text{ and }
  \liminf_{t\searrow 0} \frac{\E_\omega (\hat{u} + t\hat{v}) - \E_\omega
    (\hat{u} )}{t} \geq 0 \text{ for every } \hat{v}\in{\rm
    ker}\, j \Big\}
\end{displaymath}
for every $u\in H$ and $\omega\in \R$. Thus, if $\E_\omega$ is convex for some
$\omega\in \R$, then
\begin{displaymath}
  \hat{E}_u =\Big\{ \hat{u}\in D(\E )\;\Big\vert\; j(\hat{u})=u \text{ and }
  \E_\omega (\hat{u} + \hat{v}) - \E_\omega (\hat{u} ) \geq 0 \text{ for
    every } \hat{v}\in{\rm
    ker}\, j \Big\}
\end{displaymath}
for every $u\in H$ and by using the fact that for every $\hat{v}\in\kernel j$,
\begin{equation}
  \label{eq.elliptic.extension}
    \E_\omega (\hat{u} + \hat{v})  = \E (\hat{u} + \hat{v})
    + \frac{\omega}{2} \, \| j(\hat{u} + \hat{v})\|_\HH^2 
     = \E (\hat{u} + \hat{v}) + \frac{\omega}{2} \, \| j(\hat{u})
    \|_\HH^2,
\end{equation}
we can conclude that if $\E$ is $j$-semiconvex then
\begin{equation}
  \label{eq:set-of-extension}
  \hat{E}_{u}
  = \Big\{ \hat{u}\in D(\E )\;\Big\vert\; j(\hat{u})=u
  \text{ and } \E (\hat{u} + \hat{v}) - \E (\hat{u} ) \geq 0 \text{ for
    every } \hat{v}\in{\rm ker}\, j \Big\}
\end{equation}
for every $u\in H$.

On the one hand, the set $\hat{E}_{u}$ is motivated by the definition of
the $j$-subgradient $\partial_j\E$. In fact, if $(u,f)\in\partial_j\E$,
and if $\hat{u}\in D(\E )$ is such that $j(\hat{u}) = u$ and
\begin{displaymath}
  \liminf_{t\searrow 0} \frac{\E (\hat{u} + t\hat{v}) -
    \E (\hat{u} )}{t} \geq \langle f,j(\hat{v})\rangle_\HH
  \text{ for every } \hat{v}\in\V ,
\end{displaymath}
as in the definition of $\partial_j\E$, then 
$\hat{u}$ is necessarily an elliptic extension of $u$. Hence, 
\begin{displaymath}
  \partial_j \E  = \Bigg\{ (u,f)\in\HH\times\HH \;\Bigg\vert\;
  \begin{array}[c]{c}
  \exists \hat{u}\in\hat{E}_u \text{ such that for every } \hat{v}\in\V \\[0,1cm]
   \liminf_{t\searrow 0} \frac{\E
    (\hat{u} + t\hat{v}) -\E (\hat{u})}{t} \geq
  \langle f,j(\hat{v})\rangle_\HH
\end{array}
 \bigg\} 
\end{displaymath}
and if $\E_\omega$ is convex for some $\omega\in \R$, then we obtain in
a similar manner to claim~(a) of Lemma~\ref{lem.easy.ident} that
\begin{equation*}
 \partial_j\E  = \Bigg\{ (u,f)\in\HH\times\HH \;\Bigg\vert
 \begin{array}[c]{c}
   \exists\;\hat{u}\in\hat{E}_u \text{ such that for every } 
   \hat{v}\in\V \\[0,1cm] 
  \E_\omega (\hat{u} + \hat{v}) -\E_\omega (\hat{u}) \geq 
  \langle f+\omega j(\hat{u}),j(\hat{v})\rangle_\HH
 \end{array}
 \Bigg\} .
\end{equation*}
In other words, for the identification of the $j$-subgradient
$\partial_j\E (u)$ at a point $u\in H$ we only need to consider elliptic
extensions $\hat{u}\in \hat{E}_{u}$ of $u$ (instead of general
$\hat{u}\in D(\E )$).

Often, these elliptic extensions are obtained as solutions of an
elliptic problem with input data $u$, explaining why we call them {\em
  elliptic extensions}; compare also with Caffarelli and
Silvestre~\cite{CafSilv07}, where this notion was used in a similar
situation.

\begin{lemma}
  \label{lem:energy-constant}
  Let $\V$, $\HH$, $j$ and $\E$ be as above. Then:
\begin{itemize}
\item[(a)] If, for some $\omega\in\R$, the functional $\E_\omega$ is
  convex (resp. coercive, resp. lower semicontinuous), then for every
  $\hat{u}\in\V$ the restriction $\E|_{\hat{u} + \kernel j}$ is convex
  (resp. coercive, resp. lower semicontinuous).
\item[(b)] If $\E$ is $j$-semiconvex  and $u = j(\hat{u})$ for
  some $u\in\HH$ and $\hat{u}\in\V$, then
  \begin{displaymath}
    \hat{E}_u = \Big\{ \hat{v}\in \hat{u}+\kernel j\;\vert\; \hat{v} 
    \text{ minimises } \E|_{\hat{u} + \kernel j} \Big\} .
  \end{displaymath}

\item[(c)] If $\E$ is $j$-semiconvex,  then for every $u\in
  D(\partial_j\E )$ and every $\hat{u}\in \hat{E}_u$ one has
  \begin{displaymath}
    \E (\hat{u} ) = \inf_{j(\hat{v}) = u} \E (\hat{v})
  \end{displaymath}
  In particular, $\E$ is constant on $\hat{E}_u$ for every $u\in H$.
\end{itemize}
\end{lemma}

\begin{proof}
  Claim~(a) follows from the trivial
  observation~\eqref{eq.elliptic.extension}, (b) directly
  from \eqref{eq:set-of-extension}, and (c) follows from (b).
\end{proof}

\subsection{Identification of $\Eh$} \label{sec.identification}

We shall now identify the functional $\Eh$ from Corollaries
\ref{cor:extended-convex-functional} and
\ref{cor.extended-quasiconvex-functional} (only up to a constant, of
course). Throughout this section, $\E$ is assumed to be proper and
$j$-semiconvex. For the identification in Theorem~\ref{thm.eh.identification} 
we will asume in addition that $\E$ is in fact convex, lower semicontinuous 
and $j$-elliptic.

Consider first the two functionals $\E_0$, $\E_1 : H\to\eR$ given by
\begin{align*}
  \E_0 (u) & := \inf_{j(\hat{u} ) = u} 
  \E (\hat{u} ) , \text{ and} \\
  \E_1 (u) & := \sup_{U\subseteq \HH
    \text{ open} \atop u\in U} \inf_{j(\hat{v}) \in U} \E (\hat{v})
  \quad (u\in\HH ) .
\end{align*}
By definition of $\E_0$ and by definition of the $j$-subgradient,  
\begin{equation} \label{domain.e5}
 D(\E_0 ) = j(D(\E )) \supseteq D(\partial_j \E ) ,
\end{equation}
and in particular $\E_0 (u)$ is finite for every $u\in D(\partial_j\E )$. Now choose
$(u_0,f_0)\in\partial_j\E$, and consider in addition the functionals
$\E_2$, $\E_3 : H \to\eR$ given by
\begin{align*}
  \E_2 (u) & := \sup \Bigg\{ 
  \sum_{i=0}^{n} \langle f_i , u_{i+1} - u_i \rangle_{\HH}
    + \E_0 (u_0) \Bigg\vert
    \begin{array}[c]{c}
      n\in\N,\;(u_i,f_i) \in\partial_j\E\\
      \text{ for $i=1,\dots,n$, $u_{n+1}=u$}
    \end{array}
    \Bigg\}\\
  \E_3 (u) & := \sup \Big\{ \langle f,u-v\rangle_\HH
  + \E_0 (v) \Big\vert (v,f)\in \partial_j\E \Big\} .
\end{align*}
Note that formally the definition of the functional $\E_2$ depends on
the choice of the pair $(u_0,f_0)$. However, under somewhat stronger 
but natural assumptions on $\E$ it is in fact independent of this choice.  

\begin{theorem}[Identification of $\Eh$ for convex $\E$] \label{thm.eh.identification}
  Assume that $\E$ is convex, proper, lower semicontinuous and 
  $j$-elliptic, and let $\Eh$ be the functional from Corollary
  \ref{cor:extended-convex-functional}. Then we have
  \begin{displaymath}
    \Eh = \E_0 = \E_1 = \E_2 = \E_3 ,
  \end{displaymath}
  where the first equality holds modulo an additive constant, and
  \begin{displaymath}
    D(\Eh ) = j(D(\E )) .
  \end{displaymath}
\end{theorem}

\begin{proof} {\em $1$st step.} We claim that the functionals $\E_1$,
  $\E_2$, and $\E_3$ are convex and lower semicontinuous. The functionals
  $\E_2$ and $\E_3$ are convex and lower semicontinuous because they are
  pointwise suprema of families of continuous, convex functionals. In order 
  to see that $\E_1$ is lower semicontinuous, we show that the superlevel sets $\{
  \E_1 > c\}$ are open for every $c\in\R$. If $c\in\R$ and if $u\in \{
  \E_1 > c\}$, then, by definition of the supremum, there exists an open
  neighbourhood $U$ of $u$ such that
  \begin{displaymath}
    \inf_{j(\hat{v})\in U} \E (\hat{v} ) > c .
  \end{displaymath}
  However, by definition of $\E_1$, this means $U\subseteq \{ \E_1
  >c\}$. Hence, the superlevel set $\{ \E_1 >c\}$ is open for every
  $c\in\R$, and $\E_1$ is lower semicontinuous. Convexity of $\E_1$ is
  shown by restricting the supremum in the definition of $\E_1$ to the
  supremum over {\em convex}, open neighbourhoods $U$ of the origin $0$,
  by replacing the infimum over all $\hat{v}\in V$ satisfying
  $j(\hat{v})\in U$ with the infimum over all $\hat{v}\in V$ satisfying
  $j(\hat{v})\in u+U$, and by using a similar argument as for $\E_2$ and
  $\E_3$.\\[3pt]
\noindent {\em $2$nd step.}
 We prove that 
\[
 \E_0 = \E_1 .
\]
The inequality $\E_0 \geq \E_1$ follows immediately from the definition of both functionals. In order to prove the converse inequality, fix $u$ such that $\E_1 (u) <\infty$ (if $\E_1 (u) = \infty$, then the inequality $\E_0 (u)\leq \E_1 (u)$ is trivial). By definition of $\E_1$ and by choosing a filter of open neighbourhoods of $u$, we find a sequence $(\hat{u}_n )$ in $D(\E )$ such that
  \begin{equation}
   \label{eq:lim-to-E4}
   \begin{split}
     u & = \lim_{n\to\infty} j(\hat{u}_n ) \quad \text{and}\\
     \E_1 (u) & = \lim_{n\to\infty} \E (\hat{u}_n ) .
   \end{split}
  \end{equation}
By assumption, there exists $\omega\geq 0$ such that $\E_\omega$ is
  lower semicontinuous and coercive. The preceding two equalities imply
  that $(\E_\omega (\hat{u}_n ))$ is a convergent and thus bounded
  sequence in $\R$. By coercivity, there exists a weakly convergent
  subnet $(\hat{u}_\alpha )$ of $(\hat{u}_n)$. 
  Let $\hat{u}$ be its weak limit point. Since $j$ is weak-to-weak
  continuous, we have $j(\hat{u}) = u$. By definition of $\E_0$, since $\E$ is lower
  semicontinuous, also with respect to the weak topology, and by the
  second limit in~\eqref{eq:lim-to-E4}, we obtain
  \begin{equation*}
    \E_0 (u) \leq \E (\hat{u} ) \leq \liminf_{\alpha} \E(\hat{u}_\alpha )= \E_1(u)
    <\infty .
  \end{equation*}

 \noindent {\em $3$rd step.} We show that 
\begin{align}
\label{step3.1} \E_0 (u) & = \E_3 (u) \text{ for every } u\in D(\partial_j\E ) , \\
\label{step3.2} \E_0 (u) & \geq \langle f,u-v\rangle + \E_0 (v) \text{ for every } u\in H, (v,f)\in \partial_j\E , \text{ and} \\
\label{step3.3} \E_3 (u) & \geq \langle f,u-v\rangle + \E_3 (v) \text{ for every } u\in H, (v,f)\in \partial_j\E .
\end{align}
Fix $u\in D(\partial_j\E )$. The inequality $\E_3 (u) \geq \E_0 (u)$
follows by taking $v=u$ in the supremum in the definition of
$\E_3$. Now, let $u\in D(\E_{0})$ and $(v,f)\in \partial_j\E$. By the
definition of the $j$-subgradient and by Lemma
\ref{lem:energy-constant} (c), for every $\hat{v}\in\hat{E}_v$ and
every $\hat{u}\in V$ with $j(\hat{u}) = u$,
\[
 \E (\hat{u} ) \geq \langle f,u-v\rangle + \E (\hat{v})
 = \langle f,u-v\rangle + \E_0 (v) .
\]
Taking the infimum on the left-hand side of this inequality over all $\hat{u}\in V$ with $j(\hat{u}) = u$, we obtain \eqref{step3.2}. Taking then the supremum on the right-hand side of the inequality \eqref{step3.2} over all $(v,f)\in\partial_j\E$, we obtain
\begin{equation}
  \label{eq:3}
 \E_0 (u) \geq \E_3 (u).
\end{equation}
Since $D(\partial_j\E )\subseteq D(\E_{0})$ (see \eqref{domain.e5}), we obtain that
equality \eqref{step3.1} holds for $u\in D(\partial_j\E )$. The
inequality \eqref{step3.3} follows from the definition of $\E_3$ and inequality~\eqref{eq:3}.

\noindent {\em $4$th step.} We have
\begin{align*}
 \partial_j \E & \subseteq \partial\Eh , \\
 \partial_j \E & \subseteq \partial\E_0 , \\
 \partial_j \E & \subseteq \partial\E_2 , \text{ and} \\
 \partial_j \E & \subseteq \partial\E_3 .
\end{align*}
The first inclusion follows from Corollary \ref{cor:extended-convex-functional}, and the third inclusion from the proof of~\cite[Th\'eor\`eme 2.5, p.38]{Br73} and Lemma \ref{thm:gradient-cyclic}. The second and the fourth inclusion follow from \eqref{step3.2} and \eqref{step3.3}, respectively. By Theorem \ref{thm:max-monoton-partial-j-E}, the $j$-subgradient on the left-hand side of these four inclusions is maximal monotone, that is, it has no proper monotone extension. On the other hand, the subgradients on the right-hand sides are monotone by Step 1, Step 2 and Lemma \ref{thm:gradient-cyclic}. We thus conclude that
  \begin{equation*}
    \partial_j\E = \partial\Eh = \partial\E_2 = \partial\E_3 = \partial\E_0 (=\partial\E_1 ).
  \end{equation*}
  Since the functions $\Eh$, $\E_{0} = \E_1$, $\E_2$ and $\E_3$ are convex, proper and
  lower semicontinuous, we can deduce by applying~\cite[Theorem 3]{Rc66}
  that
  \begin{displaymath}
    \Eh = \E_0 = \E_1 = \E_2 = \E_3
  \end{displaymath}
  modulo an additive constant. By Steps 2 and 3, and since $\E_2 (u_0) = \E_3 (u_0)$, the equalities $\E_0 = \dots = \E_3$ hold without adding a constant. The equality $D(\Eh ) = j(D(\E ))$ follows from \eqref{domain.e5}, and we have
  proved the claim.
\end{proof}

By using again the equality \eqref{eq.translated.gradient} as in the
proof of Corollary \ref{cor.extended-quasiconvex-functional}, we obtain
immediately the following corollary to Theorem
\ref{thm.eh.identification}.

\begin{corollary}[Identification of $\Eh$ for $j$-elliptic $\E$] \label{cor.eh.identification}
  Assume that $\E$ is proper, lower semicontinuous and $j$-elliptic, and let
  $\Eh$ be the functional from Corollary
  \ref{cor.extended-quasiconvex-functional}. Then one has
  \begin{displaymath}
    \Eh = \E_0 = \E_1 ,
  \end{displaymath}
  where the first equality holds modulo an additive constant, and
  \begin{displaymath}
    D(\Eh ) = j(D(\E )) .
  \end{displaymath}
\end{corollary}

\subsection{The case when $j$ is a weakly closed operator} \label{sec.j.weakly.closed}

We shall now briefly discuss a case which is formally more general than
the setting considered up to now. As before, we let $\V$ be a real locally
convex topological vector space and $\HH$ a real Hilbert space. However,
\begin{displaymath}
 j : \V \supseteq D(j) \to \HH
\end{displaymath}
is now merely a weakly closed, linear operator, that is, its \emph{graph}
\begin{displaymath}
 G(j) := \Big\{ (\hat{u} , j(\hat{u} ) ) \;\Big\vert\; \hat{u}\in D(j)\Big\}
\end{displaymath}
is weakly closed in $\V\times\HH$, which is equipped with the natural,
locally convex product topology. The definition of the $j$-subgradient
of a functional $\E : \V\to\eR$ then admits the following
straightforward generalisation:
\begin{displaymath}
  \partial_j\E  := \Bigg\{ (u,f)\in\HH\times\HH \;\Bigg\vert\;
  \begin{array}[c]{c}
    \exists \hat{u}\in\D (\E)\cap D(j)
    \text{ s.t. } j(\hat{u}) = u \text{ and for every } \hat{v}\in D(j) \\[0,1cm]
    \liminf_{t\searrow 0} \frac{\E
    (\hat{u} + t\hat{v}) -\E (\hat{u})}{t} \geq \langle
  f,j(\hat{v})\rangle_\HH
\end{array}
\Bigg\}.
\end{displaymath}
This formally more general setting can however be reduced to the
setting considered up to now; indeed, it suffices to consider the space
\[
 \bar{\V} := G(j) ,
\]
equipped with the natural, locally convex topology induced from $\V\times\HH$, the operator
\begin{align*}
 \bar{j} : \bar{V} & \to \HH , \\
 (\hat{u} , j(\hat{u})) & \mapsto j(\hat{u}) , 
\end{align*}
and the functional
\begin{align*}
 \bar{\E} : \bar{V} & \to \eR , \\
 (\hat{u} , j(\hat{u})) & \mapsto \E (\hat{u} ) .
\end{align*}
Then $\bar{\V}$ is a locally convex topological vector space, and $\bar{j}$ is weak-to-weak
continuous. Moreover, one easily verifies that
\[
 \partial_{\bar{j}} \bar{\E} = \partial_j \E ,
\]
where the subgradient on the left-hand side of this equality is the
$\bar{j}$-subgradient initially defined and studied throughout this section while the subgradient on the
right-hand side of this equality is the $j$-subgradient defined as above, when $j$ is only a
weakly closed, linear operator. Note that it may happen that $\E$ is
proper while $\bar{\E}$ is not; it is therefore convenient to replace the definition and to say that $\E$ is {\em proper} if the {\em effective domain} $D(\E ) \cap D(j)$ is non-empty. On the
other hand, we can make the following simple but useful observations.

\begin{lemma}
  Assume that $\V$, $\HH$, $j$, $\E$, $\bar{\V}$, $\bar{j}$ and
  $\bar{\E}$ are as above. If $\E$ is convex (resp.~coercive, 
  resp.~lower semicontinuous), then the same is true of $\bar{\E}$.
\end{lemma}

So up to changing the definition of properness and of effective domain,
all results on $j$-subgradients from this section remain true, and the 
same is true for the results below.

\section{Semigroups and invariance of convex sets} \label{sec.invariance}

The main results from Section \ref{sec.subgradient} and the classical
theory of evolution equations governed by subgradients imply the
following well-posedness or generation theorem, which is the starting point 
of this section.

\begin{theorem} \label{thm:generation} Let $\V$ be a real locally convex
  topological vector space, $\HH$ a real Hilbert space and $j:\V\to\HH$ a
  linear, weak-to-weak continuous operator. Let $\E :\V\to\eR$ be
  proper, lower semicontinuous and $j$-elliptic. Then for every initial value
  $u_0\in \overline{D(\Eh)} = \overline{j(D(\E))}$ the gradient
  system
  \begin{equation}
    \label{eq:14}
    \begin{cases}
      \dot u+\partial_{j}\E(u)\ni 0 & \text{on $(0,\infty )$}\\
      \hspace{1,2cm} u(0)=u_{0} &
    \end{cases}
  \end{equation}
  admits a unique solution
  \begin{displaymath}
    u\in C(\R_+ ;\HH ) \cap W^{1,\infty}_{loc} ((0,\infty );\HH )
  \end{displaymath}
  satisfying the differential inclusion \eqref{eq:14} for almost every $t\in
  (0,\infty )$. In particular, this also means $u(t)\in D(\partial_j\E )$ for
  almost every $t\in (0,\infty )$.

  Denoting by $u$ the unique solution corresponding to the initial value 
  $u_0$, setting $S(\cdot )u_0 := u$ defines a strongly continuous 
  semigroup $S = (S(t))_{t\geq 0}$ of nonlinear Lipschitz continuous 
  mappings on $\overline{D(\Eh)}$.
\end{theorem}

  We call the semigroup $S$ the semigroup {\em generated
    by $(\E ,j)$} and we write $S\sim (\E,j)$. In what follows, it will be
  convenient to assume that $S$ is always defined on the entire Hilbert
  space $\HH$. This can be achieved by replacing $S(t)$ by $S(t)P$, if
  necessary, where $P$ denotes the orthogonal projection onto the
  closed, convex subset $\overline{D(\Eh)}$ of $H$. Note that in this
  way, the semigroup $S$ is in general only strongly continuous for
  $t>0$.

\begin{proof}
  By Corollary \ref{cor.extended-quasiconvex-functional}, the
  $j$-subgradient of $\E$ is equal to the classical subgradient
  of a proper, lower semicontinuous, elliptic functional on $\HH$. Moreover, up
  to adding a multiple of the identity the subgradient is maximal
  monotone. Well-posedness of the gradient system and generation of a
  semigroup on the closure $\overline{D(\partial\Eh)}$ of
  $D(\partial\Eh)$ in $H$ follow from
  \cite[Th\'eor\`eme 3.1]{Br73} while the regularity of solutions is
  stated in \cite[Th\'eor\`eme 3.2]{Br73}. The characterisation of
  $\overline{D(\partial\Eh)}$ used in the statement follows from \cite[Proposition
  2.11]{Br73} and Theorem \ref{thm.eh.identification}.
\end{proof}

In the context of gradient systems governed by $j$-subgradients, one might be interested in 
the lifting of solutions with values in the reference Hilbert space $\HH$ to {\em solutions} with values 
in the energy space $\V$. By a solution in the energy space we mean a function $\hat{u}:\R_+\to\V$ such that $u:= j(\hat{u})$ 
coincides almost everywhere with a solution of the gradient system \eqref{eq:14}. It is always
possible to find such a lifting, since, by Theorem \ref{thm:generation}, problem \eqref{eq:14}
admits a solution $u$ taking values in $D(\partial_j\E )$ almost everywhere. Now it suffices,
for almost every $t\in\R_+$, to choose an elliptic extension $\hat{u} (t)\in E_{u(t)} \not= \emptyset$. The 
measurability or -- in Banach spaces -- the integrability questions which arise in this
context, will not be discussed here. We only mention that if there exists $\omega\in\R$ such that 
$\E_\omega$ is {\em strictly} convex, or if $\E$ is strictly convex in each affine subspace
$\hat{v} +\kernel j$, then the sets $E_{u(t)}$ are singletons, and thus the solution $\hat{u}$
in the energy space is uniquely determined.\\

We point out that among evolution equations governed by maximal
monotone operators, gradient systems play a prominent role which is
comparable to the role of evolution equations governed by self-adjoint
linear operators among the class of all linear evolution
equations. Gradient systems exhibit a regularising effect in the sense
that the solution to an arbitrary initial value immediately moves into
the domain of the subgradient (see Theorem \ref{thm:generation}
above). Moreover, the non-autonomous gradient system
\begin{equation*}
  \begin{cases}
    \dot u+\partial_{j}\E(u)\ni f & \text{on $(0,\infty )$}\\
    \hspace{1,2cm} u(0)=u_{0} & 
  \end{cases}
\end{equation*}
has $L^2$-maximal regularity in the sense that for every initial value
$u_0\in D(\Eh ) = j(D(\E ))$ and every right-hand side $f\in L^2_{loc}
(\R_+ ;\HH )$ there exists a unique solution $u\in W^{1,2}_{loc}(\R_+ ;\HH )$
satisfying the differential inclusion almost everywhere \cite[Th\'eor\`eme 3.6]{Br73}. These 
well-known facts are fundamental for the corresponding solution theory, 
but are not the central focus of the present article.\\
  
The purpose of the rest of this section is to collect some qualitative
results for the semigroup $S$ generated by $(\E,j)$ under the
additional assumption that the energy functional $\E$ is convex. In this
case, $S$ is a semigroup of contractions \cite[Th\'eor\`eme
3.1]{Br73}. We first characterise invariance of closed, convex sets
under the semigroup generated by $(\E,j)$ in terms of the
functional $\E$. We then apply this abstract result in order to
characterise positive semigroups, a comparison principle for two
semigroups, order preserving semigroups, domination of semigroups,
$L^\infty$-contractivity of semigroups and extrapolation, in the case 
when the underlying Hilbert space $\HH$ is of the form $L^2(\Sigma)$ for 
a suitable measure space $\Sigma$. Similar results are known in the 
literature for semigroups
generated by classical subgradients; see Barth\'elemy \cite{By96} 
(except for the extrapolation result), and indeed, the following results 
will be obtained as a consequence of the results in the literature 
together with our identification theorem
(Theorem \ref{thm.eh.identification}). This is, for example, the case for
the next theorem, which extends~\cite[Th\'eor\`eme~1.1]{By96}.

We say that a not necessarily densely defined, nonlinear operator $S$ on
the Hilbert space $\HH$ {\em leaves a subset $C\subseteq\HH$ invariant}
if $SC\subseteq C$. Accordingly, we say that a semigroup $S$ leaves $C$
invariant if $S(t)$ leaves $C$ invariant for every $t\geq 0$.

\begin{theorem}
  \label{thm:beurling}
  Assume that $\E$ is convex, proper, lower semicontinuous and $j$-elliptic, and let $S$ be the
  semigroup on $\HH$ generated by $(\E,j)$.  Let $C\subseteq\HH$ be a
  closed, convex set, and denote by $P_C$ the orthogonal projection of 
  $\HH$ onto $C$. Then the following assertions are equivalent:
  \begin{enumerate}  
   \item\label{thm:beurling-claim-3} The semigroup $S$ leaves $C$ invariant.
   \item\label{thm:beurling-claim-4} For every $\lambda >0$ the
     resolvent $J_\lambda$ of $\partial_j\E$ leaves $C$ invariant.
   \item\label{thm:beurling-claim-2} For every $u\in H$ one has
     \begin{displaymath}
       \Eh(P_{C}u)\le \Eh(u).
     \end{displaymath} 
   \item\label{thm:beurling-claim-1} For every $\hat{u}\in D(\E)$ there is
    a $\hat{v}\in D(\E)$ such that $P_{C} j(\hat{u}) =j(\hat{v})$ and
    \begin{displaymath}
      \E (\hat{v} ) \leq \E (\hat{u}).
    \end{displaymath} 
  \end{enumerate}
\end{theorem}

\begin{proof}
  The equivalence between the
  assertions~\eqref{thm:beurling-claim-3},~\eqref{thm:beurling-claim-4} and 
  \eqref{thm:beurling-claim-2} follows
  from~\cite[Proposition~4.5]{Br73}
  and~\cite[Th\'eor\`eme~1.1]{By96}; we wish to prove that
  \eqref{thm:beurling-claim-2} and \eqref{thm:beurling-claim-1} are
  equivalent. Without loss of generality, we may assume that the
  equalities in Theorem \ref{thm.eh.identification} hold without adding a
  constant to $\Eh$, that is, in particular, $\Eh = \E_0$.

  Suppose \eqref{thm:beurling-claim-1} holds and take
  $u\in H$ such that $\Eh(u)$ is finite (otherwise \eqref{thm:beurling-claim-2} is
  obviously true). By the characterisation of $\Eh$
  (Theorem~\ref{thm.eh.identification}), and the fact that the infimum
  in the definition of $\E_{0}$ is a minimum, there is a $\hat{u}\in
  D(\E)$ such that $j(\hat{u})=u$ and $\Eh(u)=\E_{0}(u)=\E(\hat{u})$. In
  addition, we can deduce from the hypothesis that there is a
  $\hat{v}\in D(\E)$ satisfying $j(\hat{v})=P_{C}u$ and
  \begin{displaymath}
    \E(\hat{v}) \le \E(\hat{u}). 
  \end{displaymath}
 Applying again
 Theorem~\ref{thm.eh.identification} yields
 \begin{displaymath}
   \Eh(P_C u)=\E_{0}(P_{C}u)\le \E(\hat{v}) \le\E(\hat{u})=\Eh(u) , 
 \end{displaymath}
 and so we have proved \eqref{thm:beurling-claim-2}.

 Conversely, suppose that \eqref{thm:beurling-claim-2} is true.
 Let $\hat{u}\in D(\E)$ such that
 $j(\hat{u})=u$. Then the hypothesis, Theorem~\ref{thm.eh.identification}, 
 and the fact that the infimum in
 the definition of $\E_{0}$ is a minimum imply that there
 is a $\hat{v}\in D(\E)$ such that $j(\hat{v})=P_{C}u$ and 
 \begin{displaymath}
   \E(\hat{v})=\Eh(P_{C}u)\le \Eh(u)\le \E(\hat{u}).
 \end{displaymath}
 This proves that \eqref{thm:beurling-claim-1} is true and thus
 completes the proof.
\end{proof}

The next theorem is equivalent to Theorem~\ref{thm:beurling} and
extends~\cite[Th\'eor\`eme 1.9]{By96}.

\begin{theorem}
  \label{thm:beurling-version-2}
  Assume that $\E$ is convex, proper, lower semicontinuous and $j$-elliptic, 
  and let $C_{1}$, $C_{2}\subseteq\HH$ be two closed, convex sets such that
  \begin{equation}
    \label{eq:5}
    P_{C_{2}}C_{1}\subseteq  C_{1} ,
  \end{equation}
  where, as before, $P_{C_{2}}$ denotes the orthogonal projection of
  $\HH$ onto $C_{2}$. Suppose that the semigroup $S$ generated by
  $(\E,j)$ leaves $C_{1}$ invariant. Then the following
  assertions are equivalent:
  \begin{enumerate}
  \item\label{thm:beurling-version-2-claim-3} $S(t) (C_{1}\cap
    C_{2})\subseteq C_{2}$ for every $t\ge0$.
  \item\label{thm:beurling-version-2-claim-2} For every $u\in C_{1}$,
    one has
    \begin{displaymath}
      \Eh(P_{C_{2}}u)\le \Eh(u).
    \end{displaymath}
  \item\label{thm:beurling-version-2-claim-1} For every $\hat{u}\in
    D(\E)$ with $j(\hat{u})\in C_{1}$ there is a $\hat{v}\in D(\E)$ such
    that $P_{C_{2}} j(\hat{u}) =j(\hat{v})$ and
    \begin{displaymath}
      \E (\hat{v} ) \leq \E (\hat{u}).
    \end{displaymath}
\end{enumerate}
\end{theorem}

Indeed, if we take $C_{1}=H$ then we see that Theorem~\ref{thm:beurling}
is a special case of Theorem~\ref{thm:beurling-version-2}. However, with
a little bit more effort we also see that Theorem~\ref{thm:beurling}
implies Theorem~\ref{thm:beurling-version-2}.

\begin{proof}
  The equivalence between assertions
  \eqref{thm:beurling-version-2-claim-3}
  and~\eqref{thm:beurling-version-2-claim-2} follows
  from~\cite[Th\'eor\`eme~1.9]{By96} and the equivalence between
  \eqref{thm:beurling-version-2-claim-2}
  and~\eqref{thm:beurling-version-2-claim-1} is shown by using the same
  arguments as given in the proof of Theorem~\ref{thm:beurling}.
\end{proof}

\subsection{Positive semigroups} \label{sec.positive}

Throughout the rest of this section, $(\Sigma, {\mathcal B} ,\mu )$ is a
measure space and the underlying Hilbert space is $\HH = L^2 (\Sigma
)$. This Hilbert space is equipped with the natural ordering, the
positive cone $L^2 (\Sigma )^+$ being the set of all elements which are
positive almost everywhere, which turns it into a Hilbert lattice. The
lattice operations are denoted as usual, that is, we write $u\vee v$ and
$u\wedge v$ for the supremum and the infimum, respectively, $u^+ = u\vee
0$ is the positive part, $u^- = (-u)\vee 0$ the negative part, and $|u|
= u^+ + u^-$ the absolute value of an element $u\in L^2 (\Sigma )$.

We say that a semigroup $S$ on $L^2 (\Sigma )$ is {\em positive} if
$S(t) u \geq 0$ for every $u\geq 0$ and every $t\geq 0$. In other words,
the semigroup $S$ is positive if and only if $S$ leaves the closed 
positive cone $C:= L^2 (\Sigma )^+$ invariant. Since the positive cone
is also convex, and since the projection onto this cone is given by
\[
 P_{L^2(\Sigma )^+} u = u^+ ,
\] 
we immediately obtain from Theorem \ref{thm:beurling} the following
characterisation of positivity.

\begin{theorem}[Positive semigroups] \label{thm:positive-semigroups}
 Assume that $\E$ is convex, proper, lower semicontinuous and $j$-elliptic, that
$j(D(\E ))$ is dense in $\HH = L^2 (\Sigma )$, let $S$ be the semigroup on 
$L^2  (\Sigma )$ generated by $(\E,j)$. 
Then the following assertions are equivalent:
\begin{enumerate}
\item \label{thm:positive-semigroups-claim-3} The semigroup $S$ is
  positive.
\item \label{thm:positive-semigroups-claim-2} For every $u\in L^2
  (\Sigma )$ one has
\[
 \Eh (u^+) \leq \Eh (u) .
\]
\item \label{thm:positive-semigroups-claim-1} For every $\hat{u}\in D(\E
  )$ there is a $\hat{v}\in D(\E )$ such that $j(\hat{u})^+ =
  j(\hat{v})$ and
\[
 \E (\hat{v} ) \leq \E (\hat{u} ) .
\]
\end{enumerate}
\end{theorem}

\subsection{Comparison and domination of semigroups} \label{sec.comparison}

\begin{theorem}[Comparison of
  semigroups] \label{thm:comparison-property} Let $\V_1$ and $\V_2$ be
  two real locally convex topological vector spaces, $H=L^{2}(\Sigma)$ and let $j_{1} : \V_1\to
  L^2 (\Sigma )$ and $j_{2} : \V_2\to L^2 (\Sigma )$ be two linear
  operators which are weak-to-weak continuous. Further, let $\E_{1} :
  \V_1\to \R\cup\{+\infty\}$ and $\E_{2}: \V_2\to \R\cup\{+\infty\}$ be
  two convex, proper functionals, which are, respectively, $j_{1}$- and
  $j_{2}$-elliptic, assume that $j_1 (D(\E_1))$ and $j_2 (D(\E_2 ))$ are dense in $L^2 (\Sigma )$,
  and let $S_1$ and $S_2$ be the semigroups on $L^2 (\Sigma )$ generated
  by $(\E_{1},j_{1})$ and $(\E_{2},j_{2})$,
  respectively. In addition, suppose that $C\subseteq L^2 (\Sigma )$ is a closed,
  convex set satisfying
  \begin{equation}\label{eq:6}
    u\wedge v\in C\quad\text{and}\quad u\vee v \in C\quad
    \text{for every $u$, $v\in C$}
  \end{equation}
  and that the semigroups $S_1$ and $S_2$ leave $C$ invariant. Then the
  following assertions are equivalent:
  \begin{enumerate}
  \item\label{thm:comparison-property-claim-1} For every $u$, $v\in C$
    with $u\le v$ one has $S_1(t)u\le S_2(t)v$ for every $t\ge0$.
  
  \item\label{thm:comparison-property-claim-2} For every $u_{1}$, 
    $u_{2}\in C$ one has
    \begin{equation*}
      \label{ineq:comparison-with-Eh-on-H}
      \Eh_{1}(u_{1}\wedge u_{2})+\Eh_{2}(u_{1}\vee u_{2})
      \le \Eh_{1}(u_{1})+\Eh_{2}(u_{2}).
    \end{equation*}
  
  \item\label{thm:comparison-property-claim-3} For every $\hat{u}_1\in
    D(\E_{1})$, $\hat{u}_2\in D(\E_{2})$ with $u_1:=j_1(\hat{u}_1)\in
    C$ and $u_2:= j_2 (\hat{u}_2 )\in C$, there are $\hat{v}_1\in D(\E_{1})$,
    $\hat{v}_2\in D(\E_{2})$ such that $u_1\wedge u_2=j_1(\hat{v}_1)$,
    $u_1\vee u_2=j_2(\hat{v}_2)$ and
    \begin{equation*}
      \label{ineq:comparision-with-E-on-V}
      \E_{1}(\hat{v}_1)+\E_{2}(\hat{v}_2)
      \le \E_{1}(\hat{u}_1) + \E_{2}(\hat{u}_2).
    \end{equation*}
  \end{enumerate}
\end{theorem}

\begin{proof}
  Although the equivalence between~\eqref{thm:comparison-property-claim-1}
  and~\eqref{thm:comparison-property-claim-2} follows 
  from~\cite[Th\'e\-or\`eme~2.1]{By96}, we believe it is 
  instructive to show how this can be derived from Theorem
  \ref{thm:beurling} if one considers the product Hilbert space
  $\mathcal{H}:= L^2 (\Sigma )\times L^2 (\Sigma )$ equipped with the natural inner product,
  and the product space $\mathcal{V}:= \V_1\times \V_2$ equipped with
  the natural, locally convex product topology. Let $j :
  \mathcal{V}\to \mathcal{H}$ be the bounded linear operator and $\Phi :
  \mathcal{V} \to\eR$ the functional given respectively by
  \begin{align*}
    j (\hat{u}_1,\hat{u}_2) & := (j_{1}(\hat{u}_1) , j_{2}(\hat{u}_2))
    \text{ and}
    \\ 
    \Phi (\hat{u}_1,\hat{u}_2) & :=\E_{1} (\hat{u}_1)+\E_{2}(\hat{u}_2)
    \text{ for every }(\hat{u}_1,\hat{u}_2)\in \mathcal{V} .
  \end{align*}
  Then $\Phi$ is convex, proper, lower semicontinuous, $j$-elliptic, and the
  semigroup ${\mathcal S}$ generated by
  $(\Phi ,j)$ is just the diagonal semigroup given by
  \begin{equation}
    \label{eq:product-semigroup-S}
    {\mathcal S} (t)(u_{1},u_{2})=(S_1(t)u_{1},S_2(t)u_{2}) 
  \end{equation}
  for every $t\ge0$ and every $(u_{1},u_{2})\in D(S_1) \times
  D(S_2)$. 
  With these definitions, assertion \eqref{thm:comparison-property-claim-1}
  is equivalent to the property that the product semigroup $\mathcal{S}$ leaves
  the closed, convex set
  \begin{displaymath}
    \mathcal{C} :=\{ (u,v) \in C\times C : u\leq v\}
  \end{displaymath}
  invariant. Note that the orthogonal projection of
  $\mathcal{H}$ onto $\mathcal{C}$ is not given by
  $(u_{1},u_{1})\mapsto (u_{1}\wedge u_{2},u_{1}\vee u_{2})$, as one 
  might be led from assertion \eqref{thm:comparison-property-claim-2} to assume. However,
  if we take $\mathcal{C}_{1}=\mathcal{H}$ and $\mathcal{C}_{2} = \mathcal{C}$, 
  then by Theorem \ref{thm:beurling-version-2} and by following the same convexity
  argument as given in~\cite[p.247-250]{By96}, one sees that the property that $\mathcal{S}$ leaves
  $\mathcal{C}$ invariant and assertion \eqref{thm:comparison-property-claim-2} are equivalent.

  For us, it suffices to show that the
  assertions~\eqref{thm:comparison-property-claim-2}
  and~\eqref{thm:comparison-property-claim-3} are equivalent. So assume
  that \eqref{thm:comparison-property-claim-2} is true, and let
  $\hat{u}_{i}\in D(\E_{i})$ be such that $u_{i}:=j(\hat{u}_{i})\in C$ for
  $i=1$, $2$. By Theorem~\ref{thm.eh.identification}, $u_{i}\in
  D(\Eh_{i})$. By hypothesis, $u_{1}\wedge u_{2}\in D(\Eh_{1})$,
  $u_{1}\vee u_{2}\in D(\Eh_{2})$.
  Since the infimum in the definition of $\E_{0}$ is a minimum, it follows
  that there are $\hat{v}_{i}\in D(\E_{i})$ for $i=1,2$ such that
  $j(\hat{v}_{1})=u_{1}\wedge u_{2}$ and $j(\hat{v}_{2})=u_{1}\vee
  u_{2}$ satisfying
 \begin{displaymath}
   \E_1 (\hat{v}_{1})=\Eh_{1}(u_{1}\wedge u_{2})\qquad\text{and}\qquad
   \E_2 (\hat{v}_{2})=\Eh_{2}(u_{1}\vee u_{2}).
 \end{displaymath}
 Combining this together with the
 inequality from the hypothesis and again the
 characterisation of $\Eh$ (Theorem~\ref{thm.eh.identification}) yields
  \begin{displaymath}
    \E_1 (\hat{v}_{1})+\E_2 (\hat{v}_{2})=\Eh_1 (u_{1}\wedge u_{2}) +
    \Eh_2 (u_{1}\vee u_{2}) \le \Eh_{1}(u_{1})+\Eh_{2}(u_{2})\le
    \E_{1}(\hat{u}_{1})+\E_{2}(\hat{u}_{2}).
  \end{displaymath}
  Hence we have proved that \eqref{thm:comparison-property-claim-3} holds. 
  
  Conversely, assume that \eqref{thm:comparison-property-claim-3} is true, and
  let $u_{1}\in D(\Eh_{1})\cap C$ and $u_{2}\in D(\Eh_{2})\cap C$. Then 
  Theorem~\ref{thm.eh.identification} and the fact that the infimum in the
  definition of $\E_{0}$ is a minimum imply that there are
  $\hat{u}_{i}\in D(\E_{i})$ such that
  $j(\hat{u}_{i})=u_{i}$ and $\E_{i}(\hat{u}_{i})=\Eh(u_{i})$ for
  $i=1$, $2$. Let $\hat{v}_1\in D(\E_{1})$ and
  $\hat{v}_2\in D(\E_{2})$ be as in the hypothesis. 
  Recalling the identity $\Eh = \E_0$ from Theorem \ref{thm.eh.identification}, we obtain
\begin{align*}
 \Eh_1 (u_1\wedge u_2) + \Eh_2 (u_1\vee u_2) &\leq \E_1 (\hat{v}_1 ) + \E_2 (\hat{v}_2 ) \\
 & \leq \E_1 (\hat{u}_1 ) + \E_2 (\hat{u}_2 ) = \Eh_{1}(u_{1}) + \Eh_{2}(u_{2}) .
\end{align*}
\end{proof}

We formulate two consequences of Theorem
\ref{thm:comparison-property}. We call a semigroup $S = (S(t))_{t\geq 0}$ on $L^2 (\Sigma)$
\emph{order preserving on $C\subseteq L^2(\Sigma )$} if for every
$u$, $v\in C$ with $u\le v$ one has $S(t)u\le S(t)v$ for every $t\ge0$.
By taking the semigroup $S:=S_1=S_2$ (and $\E := \E_1 = \E_2$) in the
previous theorem, we obtain the characterisation in terms of the
functional $\E$ of the property that the semigroup $S$ is order
preserving on $C$. This extends~\cite[Corollaire~2.2]{By96}.

\begin{corollary}[Order-preserving semigroups]
  \label{cor:order-preservingness}
  Assume that $\E$ is convex, proper, lower semicontinuous and $j$-elliptic, and that 
  $j(D(\E ))$ is dense in $L^2 (\Sigma )$. 
  Suppose that $C\subseteq L^2 (\Sigma )$
  is a closed convex set satisfying~\eqref{eq:6} and that the semigroup
  $S$ on $L^2(\Sigma)$ generated by $(\E,j)$ leaves $C$ invariant. Then the
  following assertions are equivalent:
  \begin{enumerate}
  \item\label{cor:order-preservingness-claim-1} The semigroup $S$ is
    order preserving on $C$.
    
  \item\label{cor:order-preservingness-claim-2} For every $u_{1}$,
    $u_{2} \in C$ one has
    \begin{displaymath}
      \Eh(u_{1}\wedge u_{2})+\Eh(u_{1}\vee u_{2})\le \Eh(u_{1})+\Eh(u_{2}).
    \end{displaymath}

  \item\label{cor:order-preservingness-claim-3} For every $\hat{u}_1$,
    $\hat{u}_2\in D(\E )$ with $u_1:=j(\hat{u}_1)\in C$ and
    $u_2:=j(\hat{u}_2)\in C$, there are $\hat{v}_1$, $\hat{v}_2\in D(\E
    )$ such that $u_1\wedge u_2=j(\hat{v}_1)$, $u_1\vee
    u_2=j(\hat{v}_2)$ and
    \begin{displaymath}
      \E (\hat{v}_1)+\E(\hat{v}_2)\le \E(\hat{u}_1)+\E(\hat{u}_2).
    \end{displaymath}
  \end{enumerate}
\end{corollary}

Let $S_1$ and $S_2$ be two semigroups on $L^2 (\Sigma )$. We say that
the semigroup \emph{$S_{1}$ is dominated by $S_{2}$}, and we write $S_1
\preccurlyeq S_2$, if $S_{2}$ is positive and
\begin{displaymath}
  \abs{S_{1}(t)u}\le S_{2}(t)\abs{u}
\end{displaymath}
for every $u\in L^2 (\Sigma )$ and every $t\ge0$. Our next result
extends~\cite[Th\'eor\`eme~3.3]{By96}.

\begin{corollary}[Domination of semigroups]
  \label{cor:domination-property}
  Take the assumptions of Theorem \ref{thm:comparison-property}, and
  suppose that $S_{2}$ is positive and order preserving on $L^2 (\Sigma
  )^{+}$. Then the following assertions are equivalent:
  \begin{enumerate}
    \item\label{cor:domination-property-claim1} $S_{1}$ is dominated by $S_{2}$.
    
    \item\label{cor:domination-property-claim2} For every $u_{1}\in L^2 (\Sigma )$,
      $u_{2}\in L^{2}(\Sigma)^{+}$ one has
    \begin{displaymath}
      \Eh_{1}((\abs{u_{1}}\wedge u_{2})\sign (u_{1})) +
      \Eh_{2}(\abs{u_{1}}\vee u_{2})\le \Eh_1 (u_{1})+\Eh_2 (u_{2}).
    \end{displaymath}

  \item\label{cor:domination-property-claim3} For every $\hat{u}_1\in
    D(\E_1 )$ with $u_{1}:=j_{1}(\hat{u}_{1})$, $\hat{u}_2\in D(\E_2)$
    with $u_2 := j_2 (\hat{u}_2 )\in L^2 (\Sigma )^{+}$ there are
    $\hat{v}_1\in D(\E_1 )$, $\hat{v}_2\in D(\E_2 )$ such that
      \begin{displaymath}
      (\abs{u_1}\wedge u_2)\sign (u_{1}) =j_1(\hat{v}_1),\qquad
      \abs{u_1}\vee u_2=j_2(\hat{v}_2)
      \end{displaymath}
      and
      \begin{displaymath}
      \E_1 (\hat{v}_1)+\E_2 (\hat{v}_2) \le \E_1 (\hat{u}_1) + \E_2 (\hat{u}_2).
      \end{displaymath}
  \end{enumerate}
\end{corollary}


\begin{proof}[Proof of Corollary~\ref{cor:domination-property}]
  The equivalence of the assertions
  \eqref{cor:domination-property-claim1}
  and~\eqref{cor:domination-property-claim2} follows
  from~\cite[Th\'e\-or\`eme~3.3]{By96} and the equivalence between
  \eqref{cor:domination-property-claim2}
  and~\eqref{cor:domination-property-claim3} is proved by using the same
  arguments as given above in the proof of
  Theorem~\ref{thm:comparison-property}.
\end{proof}

\subsection{$L^\infty$-contractivity and extrapolation of semigroups}
\label{sec:extrapolation}

Let $\psi : \HH\to\eR$ be a convex, proper and lower semicontinuous
functional on a Hilbert space $\HH$.  We say that a maximal monotone
operator $A\subseteq \HH\times \HH$ is {\em $\psi$-accretive} if for all
$(u_{1},v_{1})$, $(u_{2},v_{2})\in A$ and all $\lambda>0$ one has
\begin{displaymath}
  \psi(u_{1}-u_{2}+\lambda (v_{1}- v_{2}))\ge \psi(u_{1}-u_{2}) .
\end{displaymath}
Similarly, we say that a semigroup $S$ on the Hilbert space $\HH$ is
{\em $\psi$-contractive}, if for all $u_1$, $u_2\in D(S)\subseteq\HH$
and all $t\geq 0$ one has
\begin{displaymath}
 \psi (S(t)u_1 - S(t)u_2) \leq \psi (u_1-u_2).
\end{displaymath}
In what follows, a family of typical examples of functionals on the Hilbert space
$\HH = L^2 (\Sigma )$ will be the $L^p$-norms (with effective domain 
$L^2 \cap L^p (\Sigma )$), and we then also speak of
$L^p$-accretivity of the operator $A$, or of $L^p$-contractivity of the
semigroup $S$.

The following result will be useful in the sequel.

\begin{lemma}{{\rm(\cite[Proposition 4.7]{Br73})}}
 \label{lem:psi-accretive}
 Let $A\subseteq \HH\times\HH$ be a maximal monotone operator on a
 Hilbert space $\HH$, and let $S$ be the semigroup generated by
 $-A$. Further, let $\psi : \HH \to\eR$ be a convex, proper and lower
 semicontinuous functional. Then $A$ is $\psi$-accretive if and only if
 $S$ is $\psi$-contractive.
 \end{lemma}

 We first characterise $L^\infty$-contractivity of semigroups. The
 equivalence of assertions \eqref{thm:Linfty-contractivity-claim-1} and
 \eqref{thm:Linfty-contractivity-claim-2} in the following theorem
 follows from Cipriani and Grillo \cite[Section 3]{CiGr03} and relies again on Theorem
 \ref{thm:beurling} and the same product semigroup construction as
 described in the proof of Theorem \ref{thm:comparison-property} (see
 also B\'enilan and Picard \cite{BePi79a} and B\'enilan and Crandall \cite{BeCr91}), while
 the proof of the equivalence of assertions
 \eqref{thm:Linfty-contractivity-claim-2} and
 \eqref{thm:Linfty-contractivity-claim-3} is similar to the proof of the
 corresponding equivalence in Theorem \ref{thm:beurling}; we omit the
 details.

\begin{theorem}[$L^\infty$-contractivity of semigroups] \label{thm:Linfty-contractivity}
  Assume that $\E$ is convex, proper, lower semicontinuous and $j$-elliptic, that $j(D(\E ))$ is
dense in $L^2 (\Sigma )$, and let $S$ be
  the semigroup on $L^2 (\Sigma )$ generated by $(\E,j)$.  
   Then the following assertions are equivalent:
\begin{enumerate}
 \item \label{thm:Linfty-contractivity-claim-1} The
   semigroup $S$ is $L^{\infty}$-contractive on $L^{2}(\Sigma)$.

  \item \label{thm:Linfty-contractivity-claim-2} For every
    $u_{1}$, $u_{2}\in\HH$ and for every $\alpha>0$, one has
    \begin{align*}
     &  \Eh\left((u_{1}\vee \tfrac{u_{1}+u_{2}-\alpha}{2})\wedge
    (\tfrac{u_{1}+u_{2}+\alpha}{2})\right) +
      \Eh\left((u_{2}\wedge \tfrac{u_{1}+u_{2}+\alpha}{2})\vee
    (\tfrac{u_{1}+u_{2}-\alpha}{2})\right) \\
     &  \quad \le \Eh ({u}_{1}) + \Eh ({u}_{2}).
    \end{align*}
 \item \label{thm:Linfty-contractivity-claim-3} For every
    $\hat{u}_1$, $\hat{u}_2\in D(\E )$ with $u_{1}=j(\hat{u}_{1})$ and 
    $u_{2}=j(\hat{u}_{2})$, and for every $\alpha>0$,
    there are $\hat{v}_{1}$, $\hat{v}_{2}\in D(\E)$ such that
    \begin{align*} 
    & \Big(u_{1}\vee \tfrac{u_{1}+u_{2}-\alpha}{2}\Big)\wedge
    \Big(\tfrac{u_{1}+u_{2}+\alpha}{2}\Big)=j(\hat{v}_{1}),\\
    & \Big(u_{2}\wedge \tfrac{u_{1}+u_{2}+\alpha}{2}\Big)\vee
    \Big(\tfrac{u_{1}+u_{2}-\alpha}{2}\Big)=j(\hat{v}_{2}),
    \end{align*}
    and
    \[
      \E\left(\hat{v}_{1}\right) +
      \E\left(\hat{v}_{2}\right)
      \le \E(\hat{u}_{1}) + \E(\hat{u}_{2}).
    \]
\end{enumerate}
\end{theorem}

If in Theorem \ref{thm:Linfty-contractivity} the semigroup $S$ is in
addition order preserving, then we obtain a large number of 
additional equivalent statements. To that end, we first 
briefly recall the notion of \emph{Orlicz spaces}. 
Regarding \cite[Chapter 3]{RaRe91}, a continuous function $\psi :
\R_+ \to\R_+$ is an {\em $N$-function} if it is convex, $\psi (s)=0$ if
and only if $s=0$, $\lim_{s\to 0+} \psi (s) / s = 0$, and
$\lim_{s\to\infty} \psi (s) / s = \infty$. Given an $N$-function $\psi$,
the {\em Orlicz space} $L^\psi(\Sigma )$ is the space
\[
L^\psi (\Sigma ) := \{ u : \Sigma \to\R \text{ measurable}: \int_\Sigma
\psi (\frac{|u|}{\alpha} )\; \mathrm{d}\mu <\infty \text{ for some }
\alpha >0\}
\]
equipped with the Orlicz-Minkowski norm
\[
\| u\|_{L^\psi} := \inf \{ \alpha >0 : \int_\Sigma \psi
(\frac{|u|}{\alpha} )\; \mathrm{d}\mu\leq 1\} .
\] 

In addition, for the following theorem, we make use
of the set $\mathcal{J}_{0}$ of all convex, lower semicontinuous functionals
 $\psi : \R\to[0,+\infty]$ satisfying $\psi (0)=0$.

 \begin{theorem}\label{thm:L1-contractivity}
  Suppose in addition to the assumptions of Theorem \ref{thm:Linfty-contractivity}
   that $j(D(\varphi))$ lies dense in $H$ and the semigroup $S$ is order preserving. Then
   the assertions \eqref{thm:Linfty-contractivity-claim-1},
   \eqref{thm:Linfty-contractivity-claim-2} and
   \eqref{thm:Linfty-contractivity-claim-3} from Theorem
   \ref{thm:Linfty-contractivity} are equivalent to each of the
   following assertions:
 \begin{enumerate} \setcounter{enumi}{3}
  \item\label{thm:L1-contractivity-claim-3}
    $\partial_{j}\E$ is $L^{\infty}$-accretive on $L^{2}(\Sigma)$.

  \item\label{thm:L1-contractivity-claim-4}
    $\partial_{j}\E$ is $L^{1}$-accretive on $L^{2}(\Sigma)$.

  \item\label{thm:L1-contractivity-claim-4-plus}
    $\partial_{j}\E$ is $L^{q}$-accretive on $L^{2}(\Sigma)$ for all
    $q\in (1,\infty)$.
  
  \item\label{thm:L1-contractivity-claim-4-plus-stern} 
    $\partial_{j}\E$ is $L^{\psi}$-accretive on $L^{2}(\Sigma)$
    for all $N$-functions $\psi$.   

  \item{\label{thm:L1-contractivity-claim4-plus-stern-alpha}}
    $\partial_{j}\E$ is completely accretive (in the sense of~\cite{BeCr91}),
    that is,
    \begin{equation}\label{ineq:completely-accretive}
      \int_{\Sigma}\psi (u_{1}-u_{2})\,\mathrm{d}\mu \le
      \int_{\Sigma} \psi (u_{1}-u_{2}+\lambda(v_{1}-v_{2}))\,\mathrm{d}\mu
    \end{equation}
    for all $\psi\in \mathcal{J}_{0}$ and all $(u_{1},v_{1})$, $(u_{2},v_{2})\in \partial_{j}\E$. 
    
  \item \label{thm:L1-contractivity-claim-5}
    The semigroup $S$ is $L^{1}$-contractive
    on $L^{2}(\Sigma)$.
  
  \item \label{thm:L1-contractivity-claim-5p}
   The semigroup $S$ is $L^{q}$-contractive
    on $L^{2}(\Sigma)$ for all $q\in (1,\infty)$.

 \item \label{thm:L1-contractivity-claim-5N}
   The semigroup $S$ is $L^{\psi}$-contractive
    on $L^{2}(\Sigma)$ for all $N$-functions $\psi$.

 \item \label{thm:L1-contractivity-claim-5-complete}
   The semigroup $S$ is {\em completely contractive}, that is, 
   \begin{equation}\label{ineq:completely-contractive-a}
     \int_{\Sigma}\psi (S(t)u_{1}-S(t)u_{2})\,\mathrm{d}\mu \le 
     \int_{\Sigma}\psi (u_{1}-u_{2})\,\mathrm{d}\mu
   \end{equation}
   for all $\psi\in \mathcal{J}_{0}$, $t\geq 0$ and all $u_{1}$, $u_{2}\in L^2(\Sigma)$.
 \end{enumerate}
 Moreover, if one of the equivalent conditions
 \eqref{thm:Linfty-contractivity-claim-1}-\eqref{thm:L1-contractivity-claim-5-complete}
 holds, and if there exists $u_0\in L^1\cap L^\infty (\Sigma )$ such
 that the orbit $S(\cdot )u_0$ is locally bounded on $\R_+$ with values
 in $L^1\cap L^\infty (\Sigma )$, then, for every $N$-function $\psi$,
 the semigroup $S$ can be extrapolated to a strongly continuous,
 order-preserving semigroup $S_\psi$ of contractions on $L^\psi(\Sigma)$.
\end{theorem}

Following the convention of \cite{CiGr03}, we call a convex, proper and lower
semicontinuous functional on $L^2 (\Sigma)$ which satisfies property
\eqref{cor:order-preservingness-claim-2} of Corollary
\ref{cor:order-preservingness} a (nonlinear) {\em semi-Dirichlet form},
and we call it a (nonlinear) {\em Dirichlet form} if it satisfies in
addition property \eqref{thm:Linfty-contractivity-claim-2} of Theorem
\ref{thm:Linfty-contractivity} above. Accordingly, we call a pair $(\E
,j)$ consisting of a weak-to-weak continuous operator $j:\V\to L^2
(\Sigma )$ and a convex, proper and $j$-elliptic functional $\E : \V\to\eR$
a {\em Dirichlet form} if it satisfies the assertions
\eqref{cor:order-preservingness-claim-3} of Corollary
\ref{cor:order-preservingness} and
\eqref{thm:Linfty-contractivity-claim-3} of Theorem
\ref{thm:Linfty-contractivity}. By Corollary
\ref{cor:order-preservingness} and Theorem
\ref{thm:Linfty-contractivity}, Dirichlet forms are exactly those energy
functionals on $L^2 (\Sigma )$ / pairs $(\E,j)$ which generate order
preserving, $L^\infty$-contractive semigroups. This characterisation goes back
to B\'enilan and Picard \cite{BePi79a}, who also used the
term {\em Dirichlet form} in the nonlinear context. B\'enilan and Picard also proved
in \cite{BePi79a} that semigroups generated by Dirichlet forms extrapolate
to contraction semigroups on all $L^q (\Sigma )$-spaces ($q\in [1,\infty ]$)
and, more generally, on Orlicz spaces; see also
\cite[Theorem 3.6]{CiGr03} for the $L^q$ case. This result is somewhat parallel to the
theory of sesquilinear Dirichlet forms; see, for example,
\cite[Corollary 2.16]{Ou04}. Theorem~\ref{thm:L1-contractivity} includes these results from \cite{BePi79a,CiGr03,Ou04}.

For the proof of Theorem~\ref{thm:L1-contractivity}, we need first the 
 so-called \emph{duality principle} for subgradients established by B\'enilan and
 Picard~\cite{BePi79a}.

 \begin{lemma}[{Duality Principle,
 \cite[Corollaire 2.1 and subsequent Example]{BePi79a}}]
 \label{lem:duality-principle}
 Let $\Eh : L^{2}(\Sigma)\to\eR$ be convex, proper
 and lower semicontinuous. Further, let $\psi : L^{2}(\Sigma)\to
 [0,\infty]$ be sublinear, proper and lower semicontinuous, and let $\hat{\psi}
 : L^{2}(\Sigma)\to [0,\infty]$ be defined by
 \begin{displaymath}
   \hat{\psi}(u)=\sup_{\psi(v)\le 1}\langle u,v\rangle_{H}
 \end{displaymath}
 for every $u\in H$. Then the subgradient $\partial\Eh$ is
 $\psi$-accretive in $L^{2}(\Sigma)$ if and only if $\partial\Eh$ is
 $\hat{\psi}$-accretive in $L^{2}(\Sigma)$. 
 \end{lemma}

 Second, we need the following nonlinear interpolation theorem
 due to B\'enilan and Crandall \cite{BeCr91}.

 \begin{lemma}[{\cite[Proposition 1.2]{BeCr91}}]
 \label{lem:nonlinear-interpolation}
 Let $M(\Sigma)$ be the space of equivalence classes
 of measurable functions $f: \Sigma\to \R$, equivalence meaning equality
 $\mu$-a.e. on $\Sigma$. Let $S : M(\Sigma ) \supseteq D(S)\to M(\Sigma)$ be an operator such that,
 for every $u$, $v \in D(S)$ and every $k\ge0$, one has either $u\wedge (v+k)\in D(S)$
 or $(u-k)\vee v\in D(S)$. Then $S$ satisfies
 \begin{equation*} 
 \int_{\Sigma}\psi (Su-Sv)\,\mathrm{d}\mu \le \int_{\Sigma} \psi (u-v)\,\mathrm{d}\mu
 \quad\text{for all $\psi\in \mathcal{J}_{0}$ and all $u$, $v\in D(S)$}
 \end{equation*}
 if and only if $S$ is order preserving and contractive for
 the $L^{1}$- and $L^{\infty}$-norms.
 \end{lemma}

Now, we can give the proof of
Theorem~\ref{thm:L1-contractivity}.

\begin{proof}[Proof of
 Theorem~\ref{thm:L1-contractivity}]
By Lemma~\ref{lem:psi-accretive},
assertion~\eqref{thm:L1-contractivity-claim-3} is equivalent to
assertion~\eqref{thm:Linfty-contractivity-claim-1} from Theorem
\ref{thm:Linfty-contractivity}, and for the same reason assertions
\eqref{thm:L1-contractivity-claim-4} and
\eqref{thm:L1-contractivity-claim-5},
\eqref{thm:L1-contractivity-claim-4-plus} and
\eqref{thm:L1-contractivity-claim-5p},
\eqref{thm:L1-contractivity-claim-4-plus-stern} and
\eqref{thm:L1-contractivity-claim-5N}, and
\eqref{thm:L1-contractivity-claim4-plus-stern-alpha} and
\eqref{thm:L1-contractivity-claim-5-complete} are equivalent. By the
duality principle (Lemma~\ref{lem:duality-principle}), assertions
\eqref{thm:L1-contractivity-claim-3} and
\eqref{thm:L1-contractivity-claim-4} are equivalent.

By Lemma \ref{lem:nonlinear-interpolation}, and by the assumption that
$S$ is order preserving, the now equivalent
assertions~\eqref{thm:Linfty-contractivity-claim-1} and
\eqref{thm:L1-contractivity-claim-5} imply the assertion
\eqref{thm:L1-contractivity-claim-5-complete}.

Now assume that assertion \eqref{thm:L1-contractivity-claim-5-complete} holds. Then the
inequality in \eqref{ineq:completely-contractive-a} holds
for every $N$-function $\psi$, as well as for every dilation
$\psi_\alpha := \psi (\frac{\cdot}{\alpha})$ of an $N$-function $\psi$
($\alpha >0$), and for all $t\geq 0$, and $u$, $v\in L^2 (\Sigma )$. In other words, if
$\psi$ is an $N$-function, then
\begin{displaymath}
 \int_{\Sigma}\psi (\frac{S(t) u-S(t)
   v}{\alpha})\,\mathrm{d}\mu
\le \int_{\Sigma} \psi (\frac{u-v}{\alpha})\,\mathrm{d}\mu
 \quad\text{for all $\alpha>0$, $t\geq 0$ and all $u$, $v\in L^2 (\Sigma )$}.
\end{displaymath}
Taking the infimum over all $\alpha >0$, we find
\begin{displaymath}
 \norm{S(t) u - S(t) v}_{L^\psi} \leq \norm{u-v}_{L^\psi}
\quad\text{for all $t\geq 0$ and all $u$, $v\in L^2 (\Sigma )$,}
\end{displaymath}
that is, the semigroup $S$ is $L^\psi$-contractive. Hence, assertion
\eqref{thm:L1-contractivity-claim-5-complete} implies assertion
\eqref{thm:L1-contractivity-claim-5N}.

The implication
\eqref{thm:L1-contractivity-claim-5N}$\Rightarrow$\eqref{thm:L1-contractivity-claim-5p}
follows by choosing $\psi (s) = s^q$ ($q\in (1,\infty )$), and the implication
\eqref{thm:L1-contractivity-claim-5p}$\Rightarrow$\eqref{thm:L1-contractivity-claim-5}
follows from a passage to the limit ($q\to 1$). We have thus proved the
equivalence of the assertions
\eqref{thm:Linfty-contractivity-claim-1}-\eqref{thm:L1-contractivity-claim-5-complete}.

Now, assume that one of the equivalent assertions
\eqref{thm:Linfty-contractivity-claim-1}-\eqref{thm:L1-contractivity-claim-5-complete}
holds, and assume that there exists $u_0\in L^1\cap L^\infty (\Sigma )$
such that the orbit $S(\cdot )u_0$ is locally bounded from $\R_+$ with
values in $L^1\cap L^\infty (\Sigma )$. The latter assumption together
with the fact that $S$ is both $L^1$-contractive and
$L^\infty$-contractive implies that for every $u_1\in L^1 \cap L^\infty
(\Sigma )$ the orbit $S(\cdot )u_1$ is locally bounded from $\R_+$ with
values in $L^1\cap L^\infty (\Sigma )$. Now let $\psi$ be an
$N$-function. Since $L^1 \cap L^\infty (\Sigma )$ is contained and dense
in $L^\psi (\Sigma )$, since the semigroup
$S$ leaves this subspace of $L^\psi (\Sigma )$ invariant, and since $S$ is 
$L^\psi$-contractive by assertion \eqref{thm:L1-contractivity-claim-5N} 
and order preserving by assumption, the semigroup $S$ extends to an 
order-preserving semigroup $S_\psi$ of contractions on $L^\psi (\Sigma )$. 
In order to see that it is strongly continuous, it
suffices to prove strong continuity on the subspace $L^1\cap L^\infty
(\Sigma )$. 

Let $u_1\in L^1\cap L^\infty (\Sigma )$. Since the orbit $S(\cdot )u_1$
is locally bounded with values in $L^1\cap L^\infty (\Sigma )$, there
exists a constant $C\geq 0$ such that
\begin{equation*}
 \sup_{t\in [0,1]} ( \| S(t)u_1\|_{L^1} + \| S(t)u_1\|_{L^\infty} ) \leq C .
\end{equation*}
Let $\varepsilon >0$. Since $\psi$ is an $N$-function, there exists
$\delta >0$ such that 
\[
 \psi (s) \leq \varepsilon \, s \text{ for every } s\in [0,\delta] .
\]
Since the function $\psi$ is bounded on $[\delta ,C]$, there exists
$C_\delta \geq 0$ such that
\[
 \psi (s) \leq C_\delta \, s^2 \text{ for every } s\in [\delta ,C] .
\]
Hence,
\begin{align*}
  \lefteqn{\limsup_{t\searrow 0} \int_\Sigma \psi (|S(t)u_1 - u_1|)\,\mathrm{d}\mu} \\
  & \leq \limsup_{t\searrow 0} \left[ \int_{|S(t)u_1-u_1|<\delta} \varepsilon  
  \, |S(t)u_1 - u_1|\,\mathrm{d}\mu + \int_{|S(t)u_1-u_1|\geq \delta} C_\delta \, 
  |S(t)u_1 - u_1|^2 \,\mathrm{d}\mu\right] \\
  & \leq \varepsilon \, \limsup_{t\searrow 0} \| S(t)u_1-u_1\|_{L^1} 
  + C_\delta \, \limsup_{t\searrow 0} \| S(t)u_1 -u_1\|_{L^2}^2 \leq \varepsilon \, 2C .
\end{align*}
Since $\varepsilon >0$ was arbitrary, we obtain
\[
 \lim_{t\searrow 0} \int_\Sigma \psi (|S(t)u_1 - u_1|)\,\mathrm{d}\mu = 0 .
\]
Replacing $\psi$ by $\psi ({\alpha}^{-1}\,\cdot\, )$ ($\alpha >0$) in this 
equality and using the definition of the $L^\psi$-norm, we deduce
\begin{displaymath}
 \lim_{t\searrow 0} \|S(t)u_1 - u_1\|_{L^\psi} = 0 .
\end{displaymath}
This completes the proof.
\end{proof}

\begin{remark}
  If we assume in Theorem~\ref{thm:L1-contractivity} that the
  underlying measure space $(\Sigma,\mu)$ is finite, then the
  semigroup $S$ is easily seen to extrapolate to a \emph{strongly
    continuous} contraction semigroup on $L^1 (\Sigma )$, too
  (contractivity holds in general and is stated in assertion
  \eqref{thm:L1-contractivity-claim-5}).

Actually, strong continuity in $L^1 (\Sigma )$ also holds for general
measure spaces, if there is an element
$u_{0}\in L^{1}\cap L^{\infty}(\Sigma)$ such that the semigroup $S$
leaves $\{u_{0}\}$
invariant. 
We only sketch the proof. Since the resolvent $J_{1}$ of
$\partial_j\E$ is $L^{1}$-contractive on $L^1 \cap L^2 (\Sigma )$ and
since by assumption, $J_{1}u_{0}\in L^{1}\cap L^{2}(\Sigma)$, the
inverse triangle inequality implies that $J_{1}$ maps $L^{1}\cap
L^{2}(\Sigma)$ into $L^{1}(\Sigma)$. Thus $J_{1}$ has a unique
extension on $L^{1}(\Sigma)$ (again denoted by $J_1$), and so the
operator $A:=J_{1}^{-1}-I$ is $m$-accretive on $L^{1}(\Sigma)$. By the
Crandall--Liggett Theorem~\cite{CraLi71}, $-A$ generates a strongly
continuous contraction semigroup on $L^{1}(\Sigma)$, which by
construction of $A$ and the concrete form of its resolvent coincides
with $S$ on $L^{1}(\Sigma)$ by the exponential formula.
\end{remark}

\section{Examples} \label{sec.examples}

\subsection{The $p$-Laplace operator with Robin boundary conditions on general open sets} \label{sec.robin}

Let $\Omega\subseteq \R^{d}$ be an open set having finite Lebesgue
measure $\abs{\Omega}<\infty$ and $1<p<\infty$. In this
example we introduce a weak formulation of the nonlinear
parabolic Robin boundary value problem
\begin{equation}
  \label{eq:parabolic-robin}
  \begin{cases}
    \hspace{0,265cm}\partial_t u-\Delta_{p}u+g(x,u)=0 & \text{in
      $(0,\infty)\times \Omega$,}\\
   \abs{\nabla u}^{p-2}\frac{\partial u}{\partial \nu}+\beta(x,u)=0 & 
   \text{on $(0,\infty)\times \partial\Omega$,} \\
   u(0,\cdot ) = u_0 & \text{in $\Omega$,}
  \end{cases}
\end{equation}
without any regularity assumptions on the boundary of $\Omega$. For similar problems
we refer the reader to Daners and Dr\'abek \cite{DaDr09} and Warma \cite{Wa14}. Throughout this
section, we assume that $g:\Omega\times\R\to\R$ and $\beta
: \partial\Omega\times\R\to\R$ satisfy the \emph{Caratheodory conditions}:
\begin{enumerate}
\item[(i)] $g(\cdot, z)$ and $\beta(\cdot, z)$ are measurable on $\Omega$
  and on $\partial\Omega$, respectively, for every $z\in \R$,

\item[(ii)]  $g(x,\cdot)$ and $\beta(y,\cdot)$ are continuous on $\R$ for 
a.e.~$x\in \Omega$ and for a.e.~$y\in \partial\Omega$, respectively.
\end{enumerate}
In addition, we assume that 
\begin{equation}
  \label{ass:f}
  \begin{cases}
  g(\cdot,0)\in L^{2}(\Omega),\text{ and the function $z\mapsto g(x,z)$ is Lipschitz
    con-} &\\
  \text{tinuous on $\R$ with constant $L\ge0$, uniformly for a.e. $x\in \Omega$,}&
  \end{cases}
\end{equation}
and
\begin{equation}
  \label{ass:beta}
  \begin{cases}
  \text{the function $z\mapsto \beta(y,z)$ is increasing on $\R$
    for} &\\
  \text{a.e. $y\in \partial\Omega$, and there are $\alpha$, $c>0$, $r\geq 1$,
    such that }\beta(y,z)z\ge\alpha\abs{z}^{r} &\\
  \text{and }\abs{\beta(y,z)}\le c\,\abs{z}^{r-1}
 \text{ for all $z\in \R$ and a.e. $y\in \partial\Omega$.}&
 \end{cases}
\end{equation}

As a first step, we consider the elliptic nonlinear Robin problem
\begin{equation}
  \label{eq:12}
  \left\{
  \begin{aligned}
    -\Delta_{p}u+g(x,u) &= f &\quad &\text{in
      $\Omega$,}\\
   \abs{\nabla u}^{p-2}\frac{\partial u}{\partial \nu}+\beta(x,u) &= 0 
   & &\text{on $\partial\Omega$,}
  \end{aligned}
  \right.
\end{equation}
where $f\in L^{q}(\Omega)$ is a given function for some $q \geq 1$ specified
below. A general approach for dealing with elliptic Robin problems on
arbitrary open sets goes back to a theory developed by Maz'ya
(cf. Daners~\cite{Da00} in the linear case $p=2$ and $g=0$), which we wish
to review briefly.

This theory is made possible by the following inequality
(see \cite{Mz60} and \cite[Cor.~2, Sec.~4.11.1, p.258]{Mz85}),
which states that if $\Omega$ has finite Lebesgue measure, and if the
parameters $1\le p$, $q$, $r<\infty$ satisfy
\begin{equation}
  \label{eq:mazyaparam}
  (d-p)r\le p (d-1)\; \text{ and }\; q\leq r d/(d-1),
\end{equation}
then there is a constant $C=C(d,p,q,r,|\Omega|)>0$ such that
\begin{equation}
  \label{ineq:mazya}
  \norm{u}_{L^{q}(\Omega)}\leq C\,
  \left(\norm{\nabla u}_{L^{p}(\Omega)^{d}}+
  \norm{u_{|\partial\Omega}}_{L^{r}(\partial\Omega)}\right) 
\end{equation}
for all $u\in W^{1,p}(\Omega)\cap C_c(\overline{\Omega})$. Here
$C_c(\overline{\Omega})$ is the set of all functions $u\in
C(\overline \Omega)$ with compact support in $\overline\Omega$, and 
$W^{1,p} (\Omega)$ is the classical Sobolev space. 
We shall refer to inequality~\eqref{ineq:mazya} as \emph{Maz'ya's inequality}. 
This inequality motivates the introduction of the following Sobolev-type spaces. 
Firstly, for $1\le p$, $q\le \infty$ let $W^{1}_{p,q}(\Omega)$ be the Banach 
space of all $u\in L^{q}(\Omega)$ having all distributional partial derivatives
$\tfrac{\partial u}{\partial x_{1}}$, $\dots$, $\tfrac{\partial
  u}{\partial x_{d}}\in L^{p}(\Omega)$. We equip $W^{1}_{p,q}(\Omega)$
with the natural norm $\norm{u}_{W^{1}_{p,q}}:=
\norm{u}_{L^{q}(\Omega)}+\norm{\nabla u}_{L^{p}(\Omega)^{d}}$.
Secondly, we define the space 
$V_{p,r} (\Omega )$ to be the abstract completion of
\begin{equation} \label{eq:v0}
  V_{0}:=\Big\{ u\in W^{1,p}(\Omega)\cap
  C_c(\overline{\Omega})\,\Big\vert\; \norm{u}_{V_{p,r}}<\infty\Big\}
\end{equation}
with respect to the norm
\begin{equation*}
  \norm{u}_{V_{p,r}}:=\norm{\nabla u}_{L^{p}(\Omega)^{d}} + 
  \norm{u_{\vert\partial\Omega}}_{L^{r}(\partial\Omega)},
\end{equation*}
where $L^{r}(\partial\Omega):=L^{r}(\partial\Omega,\mathcal{H}^{d-1})$,
and $\mathcal{H}=\mathcal{H}^{d-1}$ denotes the $(d-1)$-dimensional 
Hausdorff measure on the boundary $\partial\Omega$. (Note that in 
\cite{Mz85}, the function space $V_{p,r} (\Omega)$ is denoted by 
$W^{1}_{p,r}(\Omega,\partial\Omega)$.)

Maz'ya's inequality~\eqref{ineq:mazya} says that if $1 \leq p$, $q$, $r <\infty$ 
satisfy \eqref{eq:mazyaparam}, then the natural embedding
\begin{equation}
\label{eq:j-0}
  j_{0} : V_{0}\to W^{1}_{p,q}(\Omega), \quad u\mapsto u 
\end{equation}
is well defined and bounded. Moreover, by definition of $V_0$, the operator
\begin{displaymath}
\iota_0 : V_0 \to W^1_{p,q}(\Omega) \times L^r(\partial\Omega),\quad
	u \mapsto (u, u_{|\partial\Omega}),
\end{displaymath}
is well defined and bounded, too, and it is
an isomorphism from $V_0$ onto its image. 
The operator $\iota_0$ then has a unique extension to a bounded linear operator 
\[
\iota: V_{p,r}(\Omega) \to W^1_{p,q}(\Omega) \times L^r(\partial\Omega)
\] 
which is again an isomorphism from $V_{p,r}(\Omega)$ onto its image.
This means we may identify $V_{p,r}(\Omega)$ with a closed linear subspace of 
$W^{1}_{p,q}(\Omega) \times L^r(\partial\Omega)$. Let $p_1 : W^{1}_{p,q}(\Omega) \times L^r(\partial\Omega) \to W^{1}_{p,q}(\Omega)$ and $p_2 : W^{1}_{p,q}(\Omega) \times L^r(\partial\Omega)\to L^r(\partial\Omega)$ be the canonical coordinate projections. We then define the bounded linear operators 
\begin{equation}\label{eq:robin-j}
j := p_1 \circ \iota : V_{p,r}(\Omega) \to W^1_{p,q}(\Omega) , 
\end{equation}
and
\begin{equation}\label{eq:tr}
 \T := p_2 \circ \iota : V_{p,r}(\Omega) \to L^r (\partial\Omega) .
\end{equation}
For example, $j$ may be regarded as the embedding of $V_{p,r}(\Omega)$ into 
$W^{1}_{p,q} (\Omega)$ induced by Maz'ya's inequality. Or, in other words, $j$ is the
bounded linear extension of the natural embedding $j_0$ from \eqref{eq:j-0}. 
In an abuse of notation, we will also use $j$ to denote the map $V_{p,r}(\Omega) 
\to L^q(\Omega)$ given by $i\circ p_1 \circ \iota$, where $i:W^1_{p,q}(\Omega) 
\to L^q(\Omega)$ is the natural embedding, if there is no danger of confusion. The operator $\T$ is a natural extension of the trace operator $u\mapsto u|_{\partial\Omega}$ defined on $V_0$, and we therefore still call $\T u$ the {\em trace} of an element $u\in V_{p,r} (\Omega )$. 

\begin{remark}
\label{rem:mazya-inj}
  There is a potential complication with the map $j$ which Maz'ya did not 
  explore in~\cite{Mz60} or~\cite{Mz85}, but which has 
  subsequently received a certain amount of attention: $j$ is not necessarily 
  injective. Since an element $u$ belongs to 
  $\kernel j$ if and only if there is a sequence $(u_{n})$ in
  $W^{1,p}(\Omega)\cap C_c(\overline{\Omega})$ such that
  \begin{equation}
    \label{eq:13}
  \begin{split}
    &\lim_{n\to\infty}\nabla u_{n}=0\text{ in $L^{p}(\Omega)^{d}$,
    }\quad\lim_{n\to\infty} u_{n}=0\text{ in $L^{q}(\Omega)$,}\\
    &\hspace{2cm}\text{and }\lim_{n\to\infty} u_{n\vert\partial\Omega}= w
    \text{ in $L^{r}(\partial\Omega)$}
  \end{split}
  \end{equation}
  for some $w\in L^{r}(\partial\Omega)$, the map $j$ being injective is
  equivalent to $w=0$ whenever \eqref{eq:13} holds.  This is certainly
  true if, for example, $\Omega$ is a bounded Lipschitz domain, since in 
  that case we have a trace inequality (see, for
  instance,~\cite{Nc67}); but of course such an inequality does not 
  hold on arbitrary open sets. This important point was first raised by Daners 
  in~\cite{Da00}; soon afterwards an example of an $\Omega$ for 
  which $j$ is not injective was constructed by Warma~\cite{Wa02phd}. 
  This issue reemerged some time later when Arendt and ter Elst 
  \cite{ArEl11,ArEl12} introduced a generalisation of the 
  notion of trace valid  on an arbitrary open set, based in large part on 
  Maz'ya's inequality.
\end{remark}

There is another possible definition of trace, which is a further generalisation 
(to $p$, $q$, $r \neq 2$) of the generalisation of trace in \cite{ArEl11}.  
In particular, the following definition agrees with \cite[Section~1]{ArEl11} when
$p=q=r=2$.

\begin{definition}
\label{def:trace}
  For $1\le p$, $q$, $r\le \infty$, following \cite{ArEl11}, we say that $\varphi\in
  L^{r}(\partial\Omega)$ is a \emph{weak trace} of $u\in
  W^{1}_{p,q}(\Omega)$ if there is a sequence
  $(u_{n})$ in $W^{1}_{p,q}(\Omega)\cap C_c(\overline{\Omega})$ such that
  $u_{n} \to u$ in $W^{1}_{p,q}(\Omega)$ and
  $u_{n\vert\partial\Omega} \to \varphi$ in $L^{r}(\partial\Omega)$.
\end{definition}

In other words, $\varphi \in L^r(\partial\Omega)$ is a weak trace of 
$u \in W^{1}_{p,q}(\Omega)$ if and only if the pair $(u,\varphi) 
\in \iota(V_{p,r}(\Omega))$.

\begin{remark}
\label{rem:trace-identify}
  (a) It is known that there are domains on which functions may have 
  multiple weak traces in the sense of Definition~\ref{def:trace}; this is 
  immediately seen to be the case exactly when the map $j$ is not injective, 
  which in particular is a property of the \emph{domain} $\Omega$ and not 
  the function(s) in question.
  This can happen if $\partial\Omega$ becomes too ``disconnected'' 
  from $\Omega$ in a sense which can be made precise using the notion of 
  relative capacity; we refer to \cite{ArEl11} for more details in the case $p=
  q=r=2$. Of course, functions in $V_{p,r}(\Omega)$ always have unique traces 
  in the sense of \eqref{eq:tr}, since the map $\T$ is well defined.

  (b) Weak traces in the sense of Definition~\ref{def:trace} depend intrinsically on all 
  three parameters $p$, $q$, $r$. We expect it is possible that a given function 
  in $W^1_{p,q}(\Omega)$ may have multiple traces for some $r$ and only 
  one (or even none) for other $r$, although we do not explore this here.

  (c) If $p$, $r \geq 1$ satisfy the first inequality in \eqref{eq:mazyaparam}, 
  then one can always find a $q\geq 1$ such that $V_{p,r}(\Omega)$ 
  maps into $W^1_{p,q}(\Omega)$ (take $q=rd/(d-1)$) so that the maps 
  $j$ and $\T$ from \eqref{eq:robin-j} and \eqref{eq:tr}, respectively, are 
  well defined. Moreover, if $u\in C_c^\infty(\R^d)$, then, approximating 
  $u$ by itself wherever necessary, we may identify $u$ canonically with an 
  element of $V_{p,r}(\Omega)$ and $W^{1}_{p,q}(\Omega)$, and 
  $u_{|\partial\Omega}$ is both a trace and a weak trace of $u$.

  (d) The definition \eqref{eq:tr} of the trace of a function in $V_{p,r}
  (\Omega)$ can be easily extended to any pair $p$, $r\geq 1$, 
  even if they do not satisfy the first inequality in \eqref{eq:mazyaparam}, 
  since one can always identify $V_{p,r}(\Omega)$ canonically with a 
  closed subset of $L^p(\Omega)^d \times L^r(\partial\Omega)$; the 
  trace is simply the composition of this embedding and the projection onto $L^r(\partial\Omega)$.
  In the sequel, however, we will always 
  assume that \eqref{eq:mazyaparam} holds, and so we will tend not to
  distinguish between the various possible notions of trace.
\end{remark}

We next have a couple more results concerning the space $V_{p,r}(\Omega)$.
The following lemma is quite standard, but we state it for later use.

\begin{lemma}
  \label{lem:density}
  Let $p$, $r\ge 1$ and suppose $\mathcal{H}^{d-1}(K) < \infty$ for any 
  compact $K \subseteq \partial\Omega$. Then the set $\{ \hat u|_{\partial\Omega}:
  \hat u \in C_c^\infty(\R^d)\}$ is a subset of $\T V_{p,r}(\Omega)$ and
  is dense in $L^r(\partial\Omega)$. In particular, $\T V_{p,r}(\Omega)$ 
  is dense in $L^r(\partial\Omega)$.
\end{lemma}

\begin{proof}
By the Stone-Weierstra{\ss} theorem \cite[Chapter~0]{Yo95}, the set 
of restrictions of $C_c^\infty(\R^d)$ functions to $\partial\Omega$ is dense in 
$C_c(\partial\Omega)$. Since $\mathcal{H}^{d-1}$ is a Borel regular measure 
\cite{EvGa92}, which is finite on every compact set, $C_c(\partial\Omega)$ is 
dense in $L^r(\partial\Omega)$ (see~\cite[Theorem~3.14]{Ru74}).

Now, since we may identify every $\hat u \in C_c^\infty(\R^d)$
with an element of $V_{p,r}(\Omega)$ as in Remark~\ref{rem:trace-identify}(c) in
such a way that $\T(\hat u) = \hat u|_{\partial\Omega}$, we may
identify the set $\{ \hat u|_{\partial\Omega}: \hat u \in C_c^\infty(\R^d)\}$
with a subset of $\T V_{p,r}(\Omega)$ and conclude that $\T V_{p,r}(\Omega)$
is dense in $L^r(\partial\Omega)$.
\end{proof}

Next, we note that a density argument shows that $V_{p,r}(\Omega)$ has 
a lattice structure whose ordering is induced by that of the space $V_0$ defined in 
\eqref{eq:v0} in the natural way. We omit the easy proof, which follows directly 
from the fact that $V_0$ inherits the lattice structure of $W^{1,p}(\Omega)$ 
and $C_c(\overline\Omega)$.

\begin{lemma}\label{lem:lattice}
The space $V_{p,r}(\Omega)$ is a lattice for any $1 \leq p$, $r < \infty$, and the lattice 
operations are continuous. Moreover, 
assuming $p$, $q$, $r\geq 1$ satisfy \eqref{eq:mazyaparam}, then $\iota$ is a lattice isomorphism onto its range, and in particular $\iota(V_{p,r} 
(\Omega ))$ is a sublattice of $W^{1}_{p,q} (\Omega )\times L^r (\partial\Omega)$
equipped with its natural ordering. As a consequence, $j$ and $\T$ are lattice homomorphisms. 
\end{lemma}

%
%

With this background, we can return to studying our elliptic boundary value
problem~\eqref{eq:12}. We are now in a position to introduce the notion
of weak solutions of the elliptic Robin boundary value
problem~\eqref{eq:12} on general open sets.

\begin{definition}
  Suppose $1\le p$, $r<\infty$ satisfy~\eqref{eq:mazyaparam} with
  $q=2$ and let $f\in L^{q^{\mbox{}_{\prime}}}(\Omega)$, where $q'=\frac{q}{q-1}$.
  Then we call $\hat{u}\in V_{p,r} (\Omega )$ a \emph{weak solution} of
  the elliptic Robin boundary value problem~\eqref{eq:12} if for all
  $\hat{v}\in V_{p,r} (\Omega )$,
 \begin{displaymath}
   \int_{\Omega}\abs{\nabla
     \hat{u}}^{p-2}\nabla\hat{u}\nabla\hat{v}\,\dx +
   \int_\Omega g(x,j(\hat u)) j(\hat v)\,\dx+
   \int_{\partial\Omega}\beta(y,\T\hat{u})\T\hat{v}\,\dH
   = \int_\Omega f \, j(\hat{v}) \, \dx . 
 \end{displaymath}
\end{definition}

For all $z \in \R$, $x\in \Omega$ and $y \in\partial\Omega$ we set  
\begin{displaymath}
G(x,z)  := \int_{0}^{z} g(x,s)\,\ds \qquad \text{and}\qquad 
\mathcal{B}(y,z) := \int_{0}^{z}\beta(y,s)\,\ds . 
\end{displaymath}

\begin{lemma}
  \label{propo:energy-robin-problem}
  Suppose $p$, $r>1$ satisfy \eqref{eq:mazyaparam} with $q=2$, and let
  $j$ be the map given by \eqref{eq:robin-j}. Then the
  functional $\E : V_{p,r} (\Omega )\to \R$ defined by
  \begin{equation}
    \label{eq:20}
    \E(\hat{u}) = \tfrac{1}{p}\int_{\Omega}\abs{\nabla
    \hat{u}}^{p}\,\dx + \int_{\Omega} G(x,j(\hat{u}))\,\dx
   + \int_{\partial\Omega}\mathcal{B}(y,\T\hat{u})\,\dH \quad  
  \end{equation}
  for every $\hat{u}\in V_{p,r} (\Omega )$ is continuously differentiable
  and $j$-elliptic. Moreover, its $j$-subgradient is densely defined and given by
\begin{align*}
  &\partial_{j}\E = \Big\{ (u,f)\in L^2(\Omega )\times L^2 (\Omega)\;\Big\vert\;
  \exists\, \hat{u}\in V_{p,r}(\Omega) \text{
    s.t. }j(\hat{u})=u\text{ and }\forall\;\hat{v}\in V_{p,r}(\Omega)\Big. \\
  & \hspace{0,7cm}\Big.\int_\Omega \abs{\nabla\hat{u}}^{p-2}
  \nabla\hat{u}\nabla\hat{v}\dx + \int_\Omega g(x,j(\hat{u})) j(\hat{v}) \dx +
  \int_{\partial\Omega} \beta (y,\T\hat{u})\T\hat{v}\dH = \int_\Omega f
  j(\hat{v})\dx\Big\}.
\end{align*}
\end{lemma}

As mentioned earlier, here we commit a mild abuse of notation by considering $j$ to 
be the composite map $V_{p,r}(\Omega) \to W^1_{p,2}(\Omega) \hookrightarrow 
L^2(\Omega)$. In light of the lemma, we consider that $u \in L^2(\Omega)$ is a 
weak solution to \eqref{eq:12} for a given $f \in L^2(\Omega)$ if and only if the pair 
$(u,f) \in \partial_j\E$ as in Lemma~\ref{propo:energy-robin-problem}.

\begin{proof}
  By hypothesis on $p$ and $r$, the operator $j$ is linear and bounded from $V_{p,r} (\Omega)$ 
  into $L^{2}(\Omega)$.
  It is a standard exercise to show that the assumptions~\eqref{ass:f} on $g$ and \eqref{ass:beta} on $\beta$
  imply that $\E$ is (among other things) continuously differentiable
  (and in particular G\^ateaux differentiable) and that, for every
  $\hat{u}$, $\hat{v}\in V_{p,r} (\Omega )$,
  \begin{displaymath}
    \E'(\hat{u}) \, \hat{v} = \int_\Omega
    \abs{\nabla\hat{u}}^{p-2}\nabla\hat{u}\nabla\hat{v}\,\dx +
    \int_\Omega g(x,j(\hat{u}))j(\hat{v})\,\dx + \int_{\partial\Omega}
    \beta (y,\T\hat{u})\T\hat{v}\,\dH.
  \end{displaymath}
  Since $g$ was assumed to be Lipschitz continuous in the
  second variable, a.e. uniformly with respect to the first one
  (see~\eqref{ass:f}), for a.e. $x\in\Omega$, the real-valued function
  $z\mapsto g(x,z)+Lz$ is increasing on $\R$. Thus the primitive
  $z\mapsto G(x,z)+\frac{L}{2}z^{2}$, with $L\geq 0$ as in
  \eqref{ass:f}, is convex on $\R$ for a.e. $x\in \Omega$. It follows
  that the functional $\E_{L}$ is convex, and one easily
  verifies that $\E_{L+\varepsilon}$ is coercive for every $\varepsilon>0$.
  As a consequence, $\E$ is $j$-elliptic, and by Lemma \ref{lem.easy.ident} (b), 
  the $j$-subgradient $\partial_j\E$ takes the form as in the statement. 

  Observe that the range of $j$ contains the space of test functions
  $C_{c}^{\infty}(\Omega)$ (Remark~\ref{rem:trace-identify}(c)), and is 
  therefore dense in $L^{2}(\Omega)$, cf.~\cite[Theorem 3.14]{Ru74}). Since the 
  closure of $D(\partial\Eh ) = D(\partial_j\E )$ (Corollary \ref{cor.extended-quasiconvex-functional}) 
  and $D(\Eh) = j(V_{p,r} (\Omega ))$ (Theorem \ref{thm.eh.identification}) coincide by
  \cite[Proposition 2.11, p.39]{Br73}, we deduce that $\partial_j\E$ is densely defined.
\end{proof}

Before stating our main generation result, we introduce the following
notation. Let $S_D$ and $S_N$ be the semigroups generated by the
Dirichlet-$p$-Laplace operator and the Neumann-$p$-Laplace operator on
$\Omega$, respectively. These operators are, by definition, the
subgradients of the associated functionals $\E_D$,
$\E_N : L^2 (\Omega ) \to\eR$ given by
\begin{displaymath}
 \E_D (u) := 
 \begin{cases}
   \tfrac{1}{p}\displaystyle \int_\Omega |\nabla u|^p\,\dx & \text{if } u\in \mathring{W}^{1}_{p,2} (\Omega ) ,\\[0,4cm]
   +\infty & \text{else},
 \end{cases}
\end{displaymath}
and
\begin{displaymath}
 \E_N (u) := 
 \begin{cases}
   \tfrac{1}{p} \displaystyle\int_\Omega |\nabla u|^p\,\dx & \text{if } u\in \tilde{W}^{1}_{p,2} (\Omega ) ,\\[0,4cm]
   +\infty & \text{else},
 \end{cases}
\end{displaymath}
respectively, where $\mathring{W}^{1}_{p,2} (\Omega )$ is the closure
of $C_c^\infty(\Omega)$ in $W^{1}_{p,2} (\Omega )$, while
$\tilde{W}^{1}_{p,2} (\Omega )$ is the closure of the space
$W^{1}_{p,2} (\Omega ) \cap C_c (\overline{\Omega})$ in
$W^{1}_{p,2} (\Omega )$. In this case of course we have classical
subgradients, that is, the map $j$ is the identity map on
$L^2 (\Omega )$. It is well known, and easy to verify with the help of
the results from Section \ref{sec.invariance}, that both semigroups
$S_D$ and $S_N$ are positive, order preserving and
$L^\infty$-contractive. Moreover, $S_D\preccurlyeq S_N$.

\begin{theorem}
\label{thm:robin-semigroup}
Let $\E$ be the functional defined in~\eqref{eq:20} and suppose that
$p$, $r>1$ satisfy \eqref{eq:mazyaparam} with $q=2$. Then the operator
$-\partial_{j}\E$ generates a strongly continuous semigroup $S$ on
$L^{2}(\Omega)$. If $g(x,\,\cdot\,)$ is increasing for almost every
$x\in\Omega$, then the semigroup is a semigroup of contractions, order
preserving, $L^\infty$-contractive on $L^2(\Omega)$ and extrapolates
to an order-preserving contraction semigroup $S_q$ on $L^q(\Omega)$
for any $q\in [1, \infty]$, which is strongly continuous for
$q\in [1,\infty)$ and weak*-continuous if $q=\infty$. If, in addition,
$g(\cdot ,0)=0$, then the semigroups $S_q$ are positive. Finally, if
$g=0$, then $S_D\preccurlyeq S$.
\end{theorem}

\begin{proof}
  The generation result follows immediately from Lemma
  \ref{propo:energy-robin-problem} and Theorem~\ref{thm:generation}.

  For the rest of the proof, assume that $g$ is increasing. Then the
  functional $\E$ is convex and the semigroup is a semigroup of
  contractions \cite[Th\'eor\`eme 3.1]{Br73}. In order to show that it
  is also order preserving, we apply Corollary
  \ref{cor:order-preservingness}. Let $\hat{u}_1$,
  $\hat{u}_2\in V_{p,r} (\Omega ) = D(\E )$. Then, since $j$ is a
  lattice homomorphism,
\begin{align*}
  & j(\hat{u}_1)\wedge j(\hat{u}_2) = j(\hat{u}_1 \wedge \hat{u}_2) \text{ and } j(\hat{u}_1)\vee j(\hat{u}_2) = j(\hat{u}_1 \vee \hat{u}_2) \text{ with} \\
  & \hat{u}_1 \wedge \hat{u}_2 , \, \hat{u}_1 \vee \hat{u}_2 \in V_{p,r} (\Omega ) ,
\end{align*}
and, noting that the orderings on $V_{p,r}(\Omega)$ and $L^2(\Omega)$ are consistent (cf.~Lemma~\ref{lem:lattice}), 
\allowdisplaybreaks{
\begin{align*}
 &\E ( \hat{u}_1 \wedge \hat{u}_2 ) + \E (\hat{u}_1 \vee \hat{u}_2 )\\
  &\qquad = \tfrac{1}{p} \int_\Omega |\nabla\hat{u}_1|^p \mathds{1}_{\{ \hat{u}_1\leq \hat{u}_2\}}\,\dx + \tfrac{1}{p} \int_\Omega |\nabla \hat{u}_2|^p \mathds{1}_{\{ \hat{u}_1 > \hat{u}_2\}}\,\dx \\
 & \phantom{=} + \int_\Omega G(x,j(\hat{u}_1)) \mathds{1}_{\{ j(\hat{u}_1)\leq j(\hat{u}_2)\}}\,\dx 
 + \int_\Omega G(x,j(\hat{u}_2) )  \mathds{1}_{\{ j(\hat{u}_1) > j(\hat{u}_2)\}}\,\dx \\
 & \phantom{=} + \int_{\partial\Omega} {\mathcal B} (y,\T\hat{u}_1) \mathds{1}_{\{ \T\hat{u}_1\leq \T\hat{u}_2\}}w
 + \int_{\partial\Omega} {\mathcal B} (y,\T\hat{u}_2 )  \mathds{1}_{\{ \T\hat{u}_1 > \T\hat{u}_2\}}\,\textrm{d}\mathcal{H}\\
 & \phantom{=} + \tfrac{1}{p} \int_\Omega |\nabla\hat{u}_1|^p \mathds{1}_{\{ \hat{u}_1 > \hat{u}_2\}}\,\dx + \tfrac{1}{p} \int_\Omega |\nabla \hat{u}_2|^p \mathds{1}_{\{ \hat{u}_1 \leq \hat{u}_2\}}\,\dx \\
 & \phantom{=} + \int_\Omega G(x,j(\hat{u}_1)) \mathds{1}_{\{ j(\hat{u}_1) > j(\hat{u}_2)\}}\,\dx 
 + \int_\Omega G(x,j(\hat{u}_2) )  \mathds{1}_{\{ j(\hat{u}_1) \leq j(\hat{u}_2)\}}\,\dx \\
 & \phantom{=} + \int_{\partial\Omega} {\mathcal B} (y,\T\hat{u}_1) \mathds{1}_{\{ \T\hat{u}_1> \T\hat{u}_2\}}\,\textrm{d}\mathcal{H} + \int_{\partial\Omega} {\mathcal B} (y,\T\hat{u}_2 )  \mathds{1}_{\{ \T\hat{u}_1 \leq \T\hat{u}_2\}}\,\textrm{d}\mathcal{H} \\
 & = \E (\hat{u}_1 ) + \E (\hat{u}_2 ) .
\end{align*}
}
By Corollary \ref{cor:order-preservingness}, the semigroup is order preserving.

Next, we show that the semigroup is $L^\infty$-contractive. Let $\hat{u}_1$, $\hat{u}_2\in V_{p,r} (\Omega )$ and $\alpha>0$ a real number. Then  
\begin{displaymath}
\hat v_1 = \Big(\hat u_{1}\vee \tfrac{\hat u_{1}+\hat u_{2}-\alpha}{2}\Big)\wedge
    \Big(\tfrac{\hat u_{1}+\hat u_{2}+\alpha}{2}\Big)
=
\begin{cases}
\hat u_1 \qquad\quad &\text{if $|\hat u_1 - \hat u_2|\leq \alpha$}\\
\tfrac{\hat u_1+\hat u_2-\alpha}{2} &\text{if $\hat u_1-\hat u_2<-\alpha$}\\
\tfrac{\hat u_1+\hat u_2+\alpha}{2} &\text{if $\hat u_1-\hat u_2>\alpha$}
\end{cases}
\end{displaymath}
and
\begin{displaymath}
\hat v_2 = \Big(\hat u_{2}\wedge \tfrac{\hat u_{1}+\hat u_{2}+\alpha}{2}\Big)\vee
    \Big(\tfrac{\hat u_{1}+\hat u_{2}-\alpha}{2}\Big)
=
\begin{cases}
\hat u_2 \qquad\quad &\text{if $|\hat u_2 - \hat u_1|\leq \alpha$}\\
\tfrac{\hat u_1+\hat u_2-\alpha}{2} &\text{if $\hat u_2-\hat u_1<-\alpha$}\\
\tfrac{\hat u_1+\hat u_2+\alpha}{2} &\text{if $\hat u_2-\hat u_1>\alpha$}
\end{cases}
\end{displaymath}
are in $V_{p,r}(\Omega)$ and satisfy the first two equalities in
Theorem \ref{thm:Linfty-contractivity}, assertion
\eqref{thm:Linfty-contractivity-claim-3}, with $u_1 = j(\hat{u}_1)$
and $u_2 = j(\hat{u}_2)$; here again, we have used that $j$ is a
lattice homomorphism. It remains to check that
$\E(\hat v_1) + \E(\hat v_2) \leq \E(\hat u_1) + \E(\hat u_2)$ in
order to see that assertion \eqref{thm:Linfty-contractivity-claim-3}
of Theorem \ref{thm:Linfty-contractivity} is fulfilled. As this is an
argument analogous to the one above, we omit it. By
Theorem~\ref{thm:Linfty-contractivity}, the semigroup is
$L^\infty$-contractive. The fact that the semigroup extrapolates to
the whole scale of $L^q$-spaces now follows immediately from the
preceding two steps and Theorem \ref{thm:L1-contractivity}.

Now assume in addition that $g(\cdot ,0)=0$ almost everywhere. This
assumption and the assumption that $g(x,\cdot )$ is increasing for
almost every $x\in\Omega$ imply that the primitive $G$ is
positive. For the same reason, using assumption \eqref{ass:beta},
${\mathcal B}$ is positive, too. Now let
$\hat{u}\in V_{p,r} (\Omega ) = D(\E )$. Then
$j(\hat{u})^+ = j(\hat{u}^+)$ with $\hat{u}^+\in V_{p,r} (\Omega )$
and
\begin{align*}
  \E (\hat{u}^+ ) & = \tfrac{1}{p} \int_\Omega |\nabla \hat{u}|^p \mathds{1}_{\{ j(\hat{u}) >0\}}\dx 
  + \int_\Omega G(x,j(\hat{u})) \mathds{1}_{\{ j(\hat{u}) >0\}}\dx 
  + \int_{\partial\Omega} {\mathcal B} (y,\T \hat{u}) \mathds{1}_{\{ \T \hat{u} >0\}}
  \textrm{d}\mathcal{H} \\
                  & \leq \E (\hat{u} ) .
\end{align*}
Positivity of the semigroup now follows from Theorem \ref{thm:positive-semigroups}. 

We turn to the last statement and assume now that $g=0$. Here we shall
apply Corollary~\ref{cor:domination-property}. For the domination
$S_D \preccurlyeq S$, let $u_1 \in \mathring{W}^{1,p} (\Omega )$ and
$\hat{u}_2\in V_{p,r} (\Omega)$ with
$j(\hat{u}_2) \in L^2 (\Omega )^+$. Then
$\hat{u}_1 := (u_1,0)\in \iota (V_{p,r}(\Omega ))$, that is, with an
abuse of notation, we have found an element
$\hat{u}_1 \in V_{p,r} (\Omega )$ such that $j(\hat{u}_1) =
u_1$.
Clearly,
$(|{u}_1|\wedge j(\hat{u}_2)) \sign({u}_1) \in \mathring{W}^{1}_{p,2}
(\Omega )$, $|\hat{u}_1|\vee \hat{u}_2 \in V_{p,r} (\Omega )$, and
\begin{align*}
 \lefteqn{\E_D ((|u_1| \wedge j(\hat{u}_2)) \sign (u_1)) + \E (|\hat{u}_1|\vee \hat{u}_2)} \\
 &\qquad = \tfrac{1}{p} \int_{\Omega} |\nabla u_1|^p \, \mathds{1}_{\{|u_1|\leq j(\hat{u}_2)\}}\,\dx
 	+ \tfrac{1}{p} \int_\Omega |\nabla j(\hat{u}_2)|^p \, \mathds{1}_{\{|u_1| >j(\hat{u}_2)\}}\,\dx \\
 &\qquad \phantom{=} + \tfrac{1}{p} \int_{\Omega} |\nabla u_1|^p \, \mathds{1}_{\{ |u_1|> j(\hat{u}_2)\}}\,\dx + \tfrac{1}{p} \int_{\Omega} |\nabla j(\hat{u}_2)|^p \, \mathds{1}_{\{|u_1|\leq j(\hat{u}_2)\}}\,\dx \\
 &\qquad \phantom{=} + \int_{\partial\Omega} {\mathcal B}(y,\T \hat{u}_1) \mathds{1}_{\{|\T\hat{u}_1| \leq \T\hat{u}_2\}}\,\textrm{d}\mathcal{H}
 + \int_{\partial\Omega}  {\mathcal B}(y,\T \hat{u}_2) \mathds{1}_{\{|\T\hat{u}_1| > \T\hat{u}_2\}}\,\textrm{d}\mathcal{H} \\
 &\qquad = \tfrac{1}{p} \int_\Omega |\nabla u_1|^p\,\dx + \tfrac{1}{p} \int_\Omega |\nabla j(\hat{u}_2)|^p\,\dx
   + \int_{\partial\Omega}  {\mathcal B}(y,\T \hat{u}_2) \mathds{1}_{\{|\T\hat{u}_1| > \T\hat{u}_2\}}\,\textrm{d}\mathcal{H} \\
 &\qquad \leq \E_D (u_1) + \E (\hat{u}_2 ) .
\end{align*}
Hence, by Corollary~\ref{cor:domination-property}, $S_D$ is dominated by $S$.
\end{proof}

\begin{remark}
  The article \cite{ChWa12} characterises all positive, order
  preserving, local semigroups $S$ generated by negative subgradients
  and satisfying $S_D\preccurlyeq S\preccurlyeq S_N$. These semigroups
  are generated by realisations of the $p$-Laplace operator with
  general Robin boundary conditions which formally include the class
  of Robin boundary conditions which we consider in this
  example. However, in \cite{ChWa12}, the set $\Omega$ is supposed to
  be a Lipschitz domain. The above example shows that the first
  domination still holds under relaxed assumptions on $\Omega$. It is
  therefore a natural question as to whether the domination
  $S \preccurlyeq S_N$ also holds in our context. The decisive
  question is whether Corollary~\ref{cor:domination-property} (3)
  holds for (all functions in) the spaces $V_{p,r}(\Omega)$ and
  $\tilde{W}_{p,2}^{1}(\Omega)$, which in turn seems to depend on
  whether $V_{p,r}(\Omega)$ has certain rather subtle lattice-type
  properties. Since a technical investigation at this point would take
  us too far afield, we leave it as an open question.
\end{remark}

\subsection{The $p$-Dirichlet-to-Neumann operator} \label{sec.Dirichlet-to-Neumann}

As a weak variational problem, in a certain sense our second example
bears considerable similarity to the nonlinear Robin problem
considered above. It will also use much of the same theory, in
particular (keeping the notation from the previous section) the space
$V_{p,r} (\Omega )$ for $p$, $r\ge 1$ and the trace operator
$\T: V_{p,r}(\Omega) \to L^r (\partial\Omega)$.

However, in this case the map $j$ from our abstract theory \emph{is}
the trace $\T$, rather than the map
$V_{p,r}(\Omega) \to L^{q}(\Omega)$, meaning its non-injectivity is
intrinsic to the structure of the operator and not a consequence of
$\Omega$ having rough boundary. We shall again make minimal regularity
assumptions on $\partial\Omega$, but this approach is also new in the
case of smooth boundary (apart from the recent work \cite{Ha15}; see
also Remark~\ref{rem:smooth}). For more details on the
$p$-Dirichlet-to-Neumann operator on Lipschitz domains, we refer to
\cite{Ha15}.

Here we assume that $\Omega\subseteq \R^{d}$ is an open set of finite
Lebesgue measure $|\Omega |$ for which the topological boundary
$\partial\Omega$ has locally finite $(d-1)$-dimensional Hausdorff
measure, that is, 
\begin{displaymath}
  \mathcal{H}^{d-1} (K) < \infty \text{ for every compact } K\subseteq\partial\Omega ,
\end{displaymath}
although we expect this could be weakened.  We
assume that $g:\Omega\times\R\to\R$ is a function satisfying the
Caratheodory conditions (i) and (ii) from the previous example, as
well as the growth condition
\begin{equation}
  \label{ass:f2}
  \begin{cases}
    \text{there exists } \alpha\in L^{\frac{2d}{d+1}}(\Omega) \text{ and } C\geq 0 \text{ such that} & \\
    |g(x,z)| \leq \alpha (x) + C \, |z|^{\frac{d+1}{d-1}} \text{ for
      all } z\in\R \text{ and a.e. } x\in \Omega . &
  \end{cases}
\end{equation} 
Our principal aim is to prove well-posedness of the parabolic
initial-boundary value problem
\begin{equation}
  \label{eq:11}
  \begin{cases}
    -\Delta_{p}\hat{u}+g(x,\hat{u})=0 & \text{in $(0,\infty)\times \Omega$,}\\
     \partial_t \hat{u} + \abs{\nabla \hat{u}}^{p-2} \partial_\nu \hat{u} =0
     & \text{on $(0,\infty)\times \partial\Omega$,}\\
     \hat{u}(0,\cdot ) = u_0 & \text{on } \partial\Omega,
  \end{cases}
\end{equation}
for a given initial value $u_0\in L^2 (\partial\Omega )$. This is
closely associated with the Dirichlet-to-Neumann map $\Lambda_{p,g}$
which -- formally speaking -- maps the trace
$\T \hat{u} \in D(\Lambda_{p,g})\subseteq L^2 (\partial\Omega)$
(Dirichlet data) of a weak solution $\hat{u}$ of the elliptic problem
\begin{equation*}
    -\Delta_{p}\hat{u} + g(x,\hat{u}) = 0 \text{ in $\Omega$,}\\
\end{equation*}
to the outer $p$-normal derivative
$\abs{\nabla \hat{u}}^{p-2} \partial_\nu \hat{u}$ (Neumann data); see below.\\

As mentioned above, we consider the space $V_{p,2} (\Omega )$
introduced in the previous example (in particular, $r=2$). For general
$g$ we assume $p\geq \frac{2d}{d+1}$ so that Mazya's condition
\eqref{eq:mazyaparam} is fulfilled for $q = \frac{2d}{d-1}$; note that
$q-1 = \frac{d+1}{d-1}$ and $q' = \frac{2d}{d+1}$ are the exponents
appearing in the growth condition \eqref{ass:f2}. It will be
convenient to write $\hat{u}$ for elements in $V_{p,2} (\Omega )$,
$u= \T \hat{u}$ for their traces, and $j(\hat{u})$ for their
embeddings into $L^{\frac{2d}{d-1}} (\Omega )$. We mention that if
$g=0$, then the condition on $p$ can be relaxed to $p>1$ since in this
case we do not need the embedding $j$, as one can see from the
definitions of weak solution, $p$-Dirichlet-to-Neumann
operator and underlying energy functional. However, by definition,
elements of $\hat{u}\in V_{p,2} (\Omega )$ still admit both a natural
gradient $\nabla \hat{u}\in L^p (\Omega )$ and a trace $u = \T \hat{u}
\in L^2(\partial\Omega)$ since $V_{p,2}(\Omega)$ may be identified
with a closed subset of $L^p(\Omega)^d \times L^2(\partial\Omega)$
in a natural way; see Remark~\ref{rem:trace-identify}(d). \\

We call a function $\hat{u}: \R_+ \to V_{p,2} (\Omega )$ a {\em weak solution} of 
the problem \eqref{eq:11} if there exists a function $u\in C(\R_+ ;L^2 (\partial\Omega )) 
\cap W^{1,2}_{loc} ((0,\infty );L^2 (\partial\Omega ))$
such that $\T \hat{u} = u$ almost everywhere, $u(0) = u_0$, and for every $\hat{v}\in V_{p,2} (\Omega )$
one has
\[
 \int_\Omega |\nabla \hat{u}|^{p-2} \nabla \hat{u} \nabla \hat{v}\,\dx + \int_\Omega g(x,j(\hat{u})) j(\hat{v})\,\dx = - \int_{\partial\Omega} \partial_t u \T \hat{v}\textrm{d}\mathcal{H} 
 \text{ for a.e. } t\in\R_+ .
\]
If we define the Dirichlet-to-Neumann map $\Lambda_{p,g}$ by
\begin{align*}
 \Lambda_{p,g} & := \Big\{ (u,f)\in L^2 (\partial\Omega ) \times L^2 (\partial\Omega ) \Big\vert\; \exists\; \hat{u}\in V_{p,2} (\Omega )
  \text{ s.t. } \T \hat{u} = u \text{ and }\Big.\\
  & \hspace{1cm}\Big. \int_\Omega |\nabla \hat{u}|^{p-2} \nabla \hat{u} \nabla \hat{v} + \int_\Omega g(x,j(\hat{u})) j(\hat{v})\,\dx = \int_{\partial\Omega} f \T \hat{v}\,\textrm{d}\mathcal{H}\quad \forall\; \hat{v}\in V_{p,2} (\Omega ) \Big\}
\end{align*}
then we see that $\Lambda_{p,g}$ is a single-valued operator, and $\hat{u} : \R_+ \to V_{p,2} (\Omega )$ is a weak solution of the problem \eqref{eq:11}
if and only if $u := \T \hat{u}$ is a solution to the abstract Cauchy problem
\[
 \dot u + \Lambda_{p,g} u = 0 \text{ on } (0,\infty ) , \quad u(0) = u_0 .
\]
We show that this latter problem is in fact a gradient system and $\Lambda_{p,g}$ can be realised as the $\T$-subgradient of an appropriate functional.

\begin{lemma}
  \label{propo:energy-of-pDN-with-f}
  Suppose $g$ is a function satisfying~\eqref{ass:f2}, and that $g(x,\cdot )$
  is monotonically increasing for almost every $x\in\Omega$. Assume further $p\ge
  \frac{2d}{d+1}$ and let $G(x,z):=\int_{0}^{z}g(x,s)\,\ds$ for every
  $z\in \R$ and almost every $x\in \Omega$.
  Then the functional $\E : V_{p,2}(\Omega)\to \R$ defined by
  \begin{equation}
    \label{energy-p-DN-with-f}
    \E(\hat{u}) = \tfrac{1}{p}\int_{\Omega}\abs{\nabla \hat{u}}^{p}\,\dx
    + \int_{\Omega}G(x,j(\hat{u}))\,\dx
  \end{equation}
  for every $\hat{u}\in V_{p,2}(\Omega)$ is convex, continuously differentiable, and $\T$-elliptic. 
  Moreover, the $\T$-subgradient $\partial_{\T}\E$ of
  $\E$ is densely defined and coincides with the Dirichlet-to-Neumann map
  $\Lambda_{p,g}$. If $g=0$, then the condition on $p$ may be relaxed to $p>1$.
\end{lemma}

\begin{proof}
  It is easily checked that the functional $\E$ defined
  by~\eqref{energy-p-DN-with-f} is continuously differentiable on
  $V_{p,2}(\Omega)$, and
\[
 \E'(\hat{u} ) \hat{v} = \int_\Omega |\nabla \hat{u}|^{p-2} \nabla \hat{u} \nabla \hat{v}\,\dx 
 + \int_\Omega g(x,j(\hat{u})) j(\hat{v})\,\dx 
\]
for every $\hat{u}$, $\hat{v}\in V_{p,2} (\Omega )$. Moreover, since $g (x,\cdot )$ is monotonically increasing for almost
  every $x\in\Omega$, its primitive $G(x,\cdot )$ is convex for almost every
  $x\in\Omega$. Hence, $\E$ is convex, too. Using the definition of the space $V_{p,2} (\Omega )$ together with Maz'ya's  inequality, we see that
  the shifted functional 
  \begin{displaymath}
    \E_{\omega}(\hat{u}) := \E(\hat{u})+
    \tfrac{\omega}{2}\norm{\T \hat{u}}_{L^{2}(\partial\Omega)}^{2}
  \end{displaymath}
  is coercive for every $\omega >0$, since under our assumptions on
  $p$ and $q$, Maz'ya's inequality shows that the term
  $\int_\Omega G(x,j(\hat{u}))\,\dx$ can be controlled by the
  $V_{p,2}$-norm of $\hat{u}$. In other words, $\E$ is
  $\T$-elliptic. The equality $\partial_{\T} \E = \Lambda_{p,g}$
  follows from the identification of the Fr\'echet derivative of $\E$
  above and from Lemma~\ref{lem.easy.ident} (b).  Finally, since the
  effective domain of $\E$ is the entire space $V_{p,2} (\Omega )$,
  and since by Lemma~\ref{lem:density} the trace operator $\T$ has
  dense range in $L^2 (\partial\Omega)$, one argues similarly as in the proof of
  Lemma \ref{propo:energy-robin-problem} (using Corollary \ref{cor.extended-quasiconvex-functional},
  Theorem \ref{thm.eh.identification} and \cite[Proposition 2.11, p.39]{Br73}) 
  that the $\T$-subgradient of $\E$
  is densely defined. The case $g=0$ and $p>1$ is treated similarly.
\end{proof}

\begin{remark} \label{rem:smooth} If $\Omega$ has Lipschitz boundary
  and $g=0$, then our construction coincides with the variational
  definition of the Dirichlet-to-Neumann map associated with
  $-\Delta_{p}$ (cf.~\cite{Ha15}, for example, or \cite{ArEl11,ArEl12}
  in the linear case $p=2$).  In this case, the trace inequality
  together with Maz'ya's inequality~\eqref{ineq:mazya} implies that
  $V_{p,2}(\Omega)$ coincides with the Sobolev space
  $W^{1}_{p,2}(\Omega)$, up to an equivalent norm. Moreover,
  $\ker(\T)$ coincides exactly with $W^{1,p}_{0}(\Omega)$, the closure
  of $C_c^\infty (\Omega)$ in the $W^{1,p}$-norm.
\end{remark}

Our desired generation result now follows from Theorem~\ref{thm:generation}. 

\begin{theorem}\label{th:d-to-n-gen}
  Let $p$ and $g$ be as in Lemma~\ref{propo:energy-of-pDN-with-f}.
  Then the Dirichlet-to-Neumann operator $\Lambda_{p,g}$ generates a
  strongly continuous semigroup $S$ of contractions on
  $L^2(\partial\Omega)$. If $g=0$, then the condition on $p$ may be
  relaxed to $p>1$.
\end{theorem}

We wish to study the order properties of this semigroup, and in particular 
show that it extrapolates to $L^q(\partial\Omega)$ for $q \in [1,\infty)$. 

\begin{theorem}\label{propo:order-preserving}
  Let $p$ and $g$ be as in Lemma~\ref{propo:energy-of-pDN-with-f}.
  Then the semigroup $S$ generated by $\Lambda_{p,g}$ on
  $L^{2}(\partial\Omega)$ is order preserving and
  $L^\infty$-contractive.  If, in addition, $g(x,0)=0$ for almost
  every $x\in\Omega$, then the semigroup is also positive and
  extrapolates to a strongly continuous semigroup of contractions on
  $L^\psi (\partial\Omega )$ for every $N$-function $\psi$.  If $g=0$,
  then the condition on $p$ may be relaxed to $p>1$.
\end{theorem}

\begin{proof}
  Obviously, it is sufficient to show that under the conditions we have imposed 
  on $g$, the energy $\E$ given by~\eqref{energy-p-DN-with-f} satisfies
  the assertion \eqref{thm:positive-semigroups-claim-1} of
  Theorem~\ref{thm:positive-semigroups}, the
  assertion~\eqref{cor:order-preservingness-claim-3} of
  Corollary~\ref{cor:order-preservingness}, and the
  assertion~\eqref{thm:Linfty-contractivity-claim-3} of
  Theorem~\ref{thm:Linfty-contractivity} with $j=\T$; this follows in
  exactly the same way as in the proof of
  Theorem~\ref{thm:robin-semigroup}. For the extrapolation one applies
  Theorem \ref{thm:L1-contractivity}, by noting that if $g(x,0)=0$ for almost every 
  $x \in \Omega$, then the
  origin in $L^2 (\partial\Omega )$ is an equilibrium point for the
  semigroup, that is, $S(t)0 = 0$ for every $t\in\R_+$.
 \end{proof}

\subsection{Coupled parabolic-elliptic systems or degenerate parabolic equations governed by a $p$-Laplace operator} \label{sec.elliptic-parabolic}

Let $\hat{\Omega}\subseteq\R^d$ be a bounded domain, 
and let $\Omega \subseteq\hat{\Omega}$ be an open subset. We consider the following coupled
parabolic-elliptic system: denoting by $D(\Delta_p^{\hat{\Omega}} )$ the domain of the
Dirichlet-$p$-Laplace operator on $\hat{\Omega}$ (the subdifferential 
of the functional $\E_D$ from Section \ref{sec.robin}), and given an $f \in L^2(\Omega)$, we 
search for a function $u$ in $(0,\infty)\times\Omega$ together with an extension $\hat u$ to 
$(0,\infty)\times\hat\Omega$ satisfying
\begin{align}
\nonumber \hat{u}(t) \in D(\Delta_p^{\hat{\Omega}} ) & \text{ for almost every }
t\geq 0 , \text{ and} \\
\label{eq.parabolic.elliptic} \begin{split}
 u & = \hat{u} \text{ in } (0,\infty )\times\Omega \\
 \partial_t u - \Delta_p u & = f  \text{ in }
(0,\infty )\times\Omega ,\\
 -\Delta_p \hat{u} & = 0  \text{ in } (0,\infty ) \times
(\hat{\Omega}\setminus\Omega ) , \\
 \hat{u} & = 0  \text{ in } (0,\infty ) \times \partial\hat{\Omega} .
\end{split}
\end{align}
This is equivalent to the degenerate equation
\begin{align}
\nonumber \hat{u}(t) \in D(\Delta_p^{\hat{\Omega}} ) & \text{ for almost every }
t\geq 0 , \text{ and} \\
\label{eq.degenerate.parabolic} \begin{split}
 u & = \hat{u} \text{ in } (0,\infty )\times\Omega \\
 1_\Omega \, \partial_t (1_\Omega \hat{u}) - \Delta_p \hat{u} & = 1_\Omega \, f  \text{ in }
(0,\infty )\times\hat{\Omega} ,\\
 \hat{u} & = 0  \text{ in } (0,\infty ) \times \partial\hat{\Omega} .
\end{split}
\end{align}
In order to reformulate these problems as an abstract gradient system,
we consider the following setting:
We let 
\[
  \V := \mathring{W}^{1,p} (\hat{\Omega}) \text{ and } \HH := L^2 (\Omega ) 
\]
with
\begin{equation*}
\begin{aligned}
\E : \mathring{W}^{1,p} (\hat{\Omega} ) & \to \eR , \\
\hat{u} & \mapsto \tfrac{1}{p} \int_{\hat{\Omega}} |\nabla \hat{u}|^p\,\dx 
\end{aligned}
\end{equation*}
and
\begin{align*}
j  : \mathring{W}^{1,p} (\hat{\Omega} ) \supseteq D(j) & \to L^2 (\Omega ) , \\
 \hat{u} & \mapsto u := \hat{u}|_\Omega ,
\end{align*}
with maximal domain. Note that $j$ is a closed, and hence weakly closed, linear operator, which is actually bounded by the Sobolev embedding theorem if $p > \frac{2d}{d+2}$. With this choice we have $(u,f)\in\partial_j\E$ if and only if 
\begin{align*}
 & \text{there exists } \hat{u}\in \mathring{W}^{1,p} (\hat{\Omega} ) \text{ with } \hat{u}|_\Omega = u \in L^2 (\Omega ) \text{ such that} \\
 & \int_{\hat{\Omega}} |\nabla\hat{u}|^{p-2} \nabla\hat{u} \nabla \hat{v}\,\dx = \int_\Omega f\hat{v}\,\dx \text{ for every } \hat{v} \in \mathring{W}^{1,p} (\hat{\Omega} ) \text{ with } \hat{v}|_\Omega \in L^2 (\Omega ) ,
\end{align*}
that is, if and only if there exists an elliptic extension $\hat{u} \in D(\Delta_p^{\hat{\Omega}})$ such that
\begin{align*}
\begin{split}
\hat{u}|_\Omega & = u , \\
-\Delta_p \hat{u} & = f \text{ in } \Omega \text{ and } \\
-\Delta_p \hat{u} & = 0 \text{ in } \hat{\Omega}\setminus\Omega .
\end{split}
\end{align*}
Hence, the coupled
parabolic-elliptic problem \eqref{eq.parabolic.elliptic} or, equivalently, the
degenerate parabolic problem \eqref{eq.degenerate.parabolic} is a special case of the abstract gradient system \eqref{eq:14} for the choice of $\V$, $\HH$, $j$ and $\E$ made above. Note that
the functional $\E$ is convex and continuously differentiable on
$\mathring{W}^{1,p} (\hat{\Omega} )$. Moreover, since $\hat{\Omega}$ is bounded, the
Poincar\'e inequality implies that the functional is also coercive. As a
consequence, by Theorem \ref{thm:max-monoton-partial-j-E}, the $j$-subgradient 
$\partial_j\E$ is maximal monotone and the negative $j$-subgradient generates  
a semigroup $S=(S(t))_{t\geq 0}$ of (nonlinear) contractions on
$L^2 (\Omega )$ (see also Theorem \ref{thm:generation}). 
More can be said about this semigroup $S$. 

\begin{theorem}
The pair $(\E ,j)$ generates a strongly continuous contraction semigroup $S$ on $L^2 (\Omega )$ which is positive, order preserving, $L^\infty$-contractive, and extrapolates to a contraction semigroup on $L^q (\Omega )$ for every $q\in [1,\infty ]$, which is strongly continuous for $q\in [1,\infty )$ and weak$^*$ continuous for $q=\infty$. Moreover, $S_D \preccurlyeq S$, where $S_D$ denotes the semigroup generated by the Dirichlet-$p$-Laplace operator on $L^2 (\Omega )$ defined in Section \ref{sec.robin}. 
\end{theorem}

\begin{proof}
We have remarked above that $-\partial_j\E$ generates a semigroup $S$ of nonlinear contractions on $L^2 (\Omega )$. Note that the semigroup is defined on $L^2 (\Omega )$ since $j(D(\E ))$ contains the test functions on $\Omega$ and is thus dense in $L^2 (\Omega )$. 

For every $\hat{u}\in D(j) \subseteq \mathring{W}^{1,p} (\hat{\Omega})$ one has $\hat{u}^+ \in D(j)$, $j(\hat{u}^+) = j(\hat{u})^+$ and 
\begin{align*}
\E (\hat{u}^+) & = \tfrac{1}{p} \int_{\hat{\Omega}} |\nabla\hat{u}^+|^p\,\dx \\
& = \tfrac{1}{p} \int_{\hat{\Omega}} |\nabla\hat{u}|^p \mathds{1}_{\{ \hat{u} >0\}}\,\dx \\
& \leq   \tfrac{1}{p} \int_{\hat{\Omega}} |\nabla\hat{u}|^p\,\dx = \E (\hat{u} ) .
\end{align*}
Hence, by Theorem \ref{thm:positive-semigroups}, $S$ is positive. Also the facts that $S$ is order preserving and $L^\infty$-contractive are proved in a way similar to that already used in the proof of Theorem \ref{thm:robin-semigroup}. We omit the details.

In order to show that $S_D$ is dominated by $S$, we shall apply Corollary~\ref{cor:domination-property}. Let $u_1 \in \mathring{W}^{1,p}(\Omega)$ and $\hat{u}_2 \in \mathring{W}^{1,p}(\hat\Omega)$ with $j(\hat{u}_2) \in L^2 (\Sigma )^+$. We extend $u_1$ by zero to an element in $\mathring{W}^{1,p}(\hat\Omega)$ (which in an abuse of notation we will also call $u_1$). Then $j((|u_1|\wedge \hat{u}_2) \sign(u_1)) = (|u_1| \wedge j(\hat{u}_2)) \sign (u_1) \in \mathring{W}^{1,p} (\Omega )$, $|u_1|\vee \hat{u}_2 \in \mathring{W}^{1,p} (\hat{\Omega} )$, and, noting that ${u}_1 = 0$ on $\hat{\Omega}\setminus\Omega$,
\begin{align*}
 \lefteqn{\E_D ((|u_1| \wedge j(\hat{u}_2)) \sign (u_1)) + \E (|u_1|\vee \hat{u}_2)} \\
 &\qquad = \tfrac{1}{p} \int_{\Omega} |\nabla u_1|^p \, \mathds{1}_{\{|u_1|\leq \hat{u}_2\}}\,\dx 
 + \tfrac{1}{p} \int_\Omega |\nabla \hat{u}_2|^p \, \mathds{1}_{\{|u_1| >\hat{u}_2\}}\,\dx \\
 &\qquad \phantom{=} + \tfrac{1}{p} \int_{\hat{\Omega}} |\nabla u_1|^p \, \mathds{1}_{\{ |u_1|> \hat{u}_2\}}\,\dx + \tfrac{1}{p} \int_{\hat{\Omega}} |\nabla \hat{u}_2|^p \, \mathds{1}_{\{|u_1|\leq\hat{u}_2\}}\,\dx \\
 &\qquad = \tfrac{1}{p} \int_\Omega |\nabla u_1|^p \, \mathds{1}_{\{|u_1|\leq \hat{u}_2\}}\,\dx + \int_{\hat{\Omega}} |\nabla \hat{u}_2|^p\,\dx \\
 &\qquad \leq \E_D (u_1) + \E (\hat{u}_2 ).
\end{align*}
Hence, by Corollary~\ref{cor:domination-property}, $S_D$ is dominated by $S$.
\vanish{
For the relation $S \leq S_N$, we take $u_1 \in \mathring{W}^{1,p}(\hat\Omega)$ and $u_2 \in W^{1,p}(\Omega) = D(\E)$ arbitrary, again without loss of generality nonnegative. Since $\Omega$ has the extension property, the function $u_{1|\Omega} \vee u_2 \in W^{1,p}(\Omega)$ has an extension $w \in W^{1,p}(\hat\Omega)$. If we set $v_1:= u_1 \vee w$, then $v_1 \in \mathring{W}^{1,p}(\hat\Omega)$ since $0 \leq v_1 \leq u_1 \in \mathring{W}^{1,p}(\hat\Omega)$ and $j(v_1)=v_{1|\Omega}$ is by construction $u_{1|\Omega} \vee u_2$. We also set $v_2:= u_{1|\Omega} \wedge u_2 \in W^{1,p}(\Omega)$. Then a routine calculation as above yields
\begin{displaymath}
	\E(v_1) + \E_N(v_2) = \E(u_1) + \E_N(u_2),
\end{displaymath}
which establishes that Theorem~\ref{thm:comparison-property}(2) holds and hence $S \leq S_N$ as claimed.}
\end{proof}

\begin{remark}
(a) The domination $S\preccurlyeq S_N$ on $L^2 (\Omega )$ is not true in general.

(b) It is possible to replace in \eqref{eq.parabolic.elliptic} or \eqref{eq.degenerate.parabolic} the Dirichlet boundary conditions on the boundary of $\hat{\Omega}$ by Neumann boundary conditions if one assumes that $\hat{\Omega}$ is a bounded domain with continuous boundary $\partial\hat{\Omega}$. In this case, one puts $\V = W^{1,p} (\hat{\Omega} )$ and uses the energy functional 
\begin{equation*}
\begin{aligned}
\E : W^{1,p} (\hat{\Omega} ) & \to \eR , \\
\hat{u} & \mapsto \frac{1}{p} \int_{\hat{\Omega}} |\nabla \hat{u}|^p\,\dx .
\end{aligned}
\end{equation*}
This energy functional is clearly convex and continuously differentiable, too. The fact that it is $j$-elliptic follows from \cite[Chapter 2, Th\'eor\`eme 7.6]{Nc67} or \cite[Corollary 4.4]{Si99}. 
\end{remark}

\subsection{Coupled parabolic-elliptic systems or degenerate parabolic equations governed by a $1$-Laplace operator} \label{sec.elliptic-parabolic.1}

In this final example we consider a variant of the coupled
parabolic-elliptic system \eqref{eq.parabolic.elliptic} from the
previous example with $p=1$, that is, with the $1$-Laplace operator
formally given by $\Delta_1 := {\rm div}\, (\frac{\nabla u}{|\nabla u|})$ and
generating the so-called total variation flow. This example
illustrates in particular why it can be useful to consider general locally convex
topological vector spaces in our abstract theory. Let
$\hat{\Omega}\subseteq\R^d$ be a bounded domain with Lipschitz continuous boundary 
and let $\Omega \subseteq\hat{\Omega}$ be an open, non-empty subset. We consider the coupled
parabolic-elliptic system \eqref{eq.parabolic.elliptic} with $p=1$, that is,
\begin{align}
\nonumber \hat{u}(t) \in D(\Delta_1^{\hat{\Omega}} ) & \text{ for almost every }
t\geq 0 , \text{ and} \\
\label{eq.parabolic.elliptic.1} \begin{split}
 u & = \hat{u} \text{ in } (0,\infty )\times\Omega \\
 \partial_t u - \Delta_1 u & = f  \text{ in }
(0,\infty )\times\Omega ,\\
 -\Delta_1 \hat{u} & = 0  \text{ in } (0,\infty ) \times
(\hat{\Omega}\setminus\Omega ) . 
\end{split}
\end{align}
Here $D(\Delta_1^{\hat{\Omega}} )$ is the domain of the $1$-Laplace operator on $\hat{\Omega}$ as generator of the total variation flow. Note that here we have left out the boundary conditions from \eqref{eq.parabolic.elliptic}, which are in fact redundant in \eqref{eq.parabolic.elliptic}, since they are included in the domain of the Dirichlet-$p$-Laplace operator. Here, as well, the domain $D(\Delta_1^{\hat{\Omega}} )$ encodes certain boundary conditions which will, however, not be discussed here. The operator $\Delta_1^{\hat{\Omega}}$ may be introduced as follows. Let
\[
BV(\hat{\Omega} ) := \{ \hat{u}\in L^1 (\hat{\Omega} ) : \forall 1\leq i\leq d \exists \mu_i \in M^b (\hat{\Omega} ) \forall \hat{v}\in C^1_c (\hat{\Omega}) : \int_{\hat{\Omega}} \hat{u} \partial_i \hat{v} = - \int_{\hat{\Omega}} \hat{v}\; \mathrm{d}\mu_i \} 
\]
be the space of all functions of bounded variation, that is, the space of all functions in $L^1 (\hat{\Omega})$ for which all distributional partial derivatives exist in the space of bounded Borel measures on $\hat{\Omega}$. We define the {\em total variation}
\[
 {\rm Var}\, (\hat{u},\hat{\Omega} ) := \sup \Big\{ \int_{\hat{\Omega}} \hat{u} {\rm div}\, \hat{v} : \hat{v}\in C^1_c (\hat{\Omega} )^d , \, \| \hat{v}\|_\infty \leq 1 \Big\} ,
\]
and then the space $BV(\hat{\Omega})$ is a Banach space for the norm
\[
 \| \hat{u}\|_{BV (\hat{\Omega})} := \| \hat{u}\|_{L^1(\hat{\Omega})} + {\rm Var}\, (\hat{u},\hat{\Omega} ) .
\]
Recall that $BV(\hat{\Omega})$ is a dual space by \cite[Remark 3.12,
p. 124]{AmFuPa00}; we denote the weak$^*$ topology by $\tau_{w^*}$. By
\cite[Definition 3.11, p. 124]{AmFuPa00}, $u_h\to u$ in the weak$^*$
topology if and only if $u_h\to u$ in $L^1 (\hat{\Omega})$ and
$\int_{\hat{\Omega}} u_h {\rm div}\, v \to \int_{\hat{\Omega}} u {\rm
  div}\, v$
for every $v\in C^1 (\hat{\Omega} )^d$. The latter convergence
corresponds to weak$^*$ convergence of the partial derivatives
$\mu_{h,i}$ in $M^b (\hat{\Omega}) = C_0 (\hat{\Omega})'$.

Let $\V := BV(\hat{\Omega})$ be equipped with the weak$^*$ topology
which turns it into a locally convex topological space. By \cite[Theorem 3.10, p.64]{Ru73}, the
weak topology in $(BV (\hat{\Omega} ),\tau_{w^*})$ coincides with the
weak$^*$ topology $\tau_{w^*}$ itself. Hence, by the Banach-Alaoglu
theorem, any norm bounded set in $BV (\hat{\Omega} )$ is relatively
weakly compact in $(BV (\hat{\Omega}),\tau_{w^*})$.

Let 
\[
 \E (\hat{u}) := {\rm Var}\, (\hat{u} ,\hat{\Omega} ) \quad (\hat{u}\in BV (\hat{\Omega} )) 
\]
be the total variation functional. As a pointwise supremum of linear,
weak$^*$ continuous functions, the total variation is convex and lower
semicontinuous on $(BV (\hat{\Omega} ) ,\tau_{w^*})$. Next, consider the identity
map
\begin{align*}
\hat{j}  : BV (\hat{\Omega} ) \supseteq D(\hat{j}) & \to L^2 (\hat{\Omega} ) , \\
 \hat{u} & \mapsto \hat{u} ,
\end{align*}
with maximal domain. This map is weakly closed. Clearly, by the
definition of $\E$ and the norm in $BV (\hat{\Omega} )$, and since
$\hat{\Omega}$ is bounded, all sublevels of
$\hat{u} \mapsto \E (\hat{u}) + \| \hat{u}\|_{L^2 (\hat{\Omega})}^2$ are
norm bounded, and thus, by the theorem of Banach-Alaoglu, relatively
weakly compact in $(BV (\hat{\Omega} ),\tau_{w^*})$. As a consequence, $\E$ is
$\hat{j}$-elliptic. By the operator $\Delta_1^{\hat{\Omega}}$ we then
mean exactly the negative $\hat{j}$-subgradient
$-\partial_{\hat{j}} \E$ on $L^2 (\hat{\Omega} )$.
 
Similarly as in the previous example, we reformulate the problem \eqref{eq.parabolic.elliptic.1} as an abstract gradient system by setting in addition $\HH := L^2 (\Omega )$, and by considering the restriction map 
\begin{align*}
j  : BV (\hat{\Omega} ) \supseteq D(j) & \to L^2 (\Omega ) , \\
 \hat{u} & \mapsto u := \hat{u}|_\Omega ,
\end{align*}
with maximal domain. Again, $j$ is weakly closed. With this choice we have $(u,f)\in\partial_j\E$ if and only if 
there exists an elliptic extension $\hat{u} \in D(\Delta_1^{\hat{\Omega}} )$ such that
\begin{align*}
\begin{split}
\hat{u}|_\Omega & = u , \\
-\Delta_1 \hat{u} & = f \text{ in } \Omega \text{ and } \\
-\Delta_1 \hat{u} & = 0 \text{ in } \hat{\Omega}\setminus\Omega .
\end{split}
\end{align*}
Hence, the coupled parabolic-elliptic problem \eqref{eq.parabolic.elliptic.1} or, equivalently, the
degenerate parabolic problem \eqref{eq.degenerate.parabolic} with $p=1$ is a special case of the abstract gradient system \eqref{eq:14} for the choice of $\V$, $\HH$, $j$ and $\E$ made above. 

By combining \cite[Chapter 2, Th\'eor\`eme 7.6]{Nc67} or \cite[Corollary 4.4]{Si99} with \cite[Theorem 3.9, p. 122]{AmFuPa00}, we find that
\[
 \| \hat{u} \| := \| \hat{u}\|_{L^1 (\Omega )} + {\rm Var}\, (\hat{u},\hat{\Omega} ) 
\]
defines an equivalent norm on $BV (\hat{\Omega} )$, and from here one sees that $\E$ is also $j$-elliptic. Hence, by Theorem \ref{thm:generation}, we immediately obtain the following result. 

\begin{theorem}
The pair $(\E ,j)$ generates a strongly continuous contraction semigroup $S$ on $L^2 (\Omega )$.
\end{theorem}

{\bf Acknowledgement.} The authors would like to thank the referees for their careful and thoughtful reading of a first version of the manuscript and for their constructive remarks which helped to improve this article.



\nocite{CaPi78}
\nocite{ChWa12}
\nocite{Wa14}
\nocite{ArMa12}
\nocite{Ru74}
\nocite{Ho03I}
\nocite{CiGr03}
\nocite{Zi89}

\providecommand{\bysame}{\leavevmode\hbox to3em{\hrulefill}\thinspace}



 \def\cprime{$'$}
  \def\ocirc#1{\ifmmode\setbox0=\hbox{$#1$}\dimen0=\ht0 \advance\dimen0
  by1pt\rlap{\hbox to\wd0{\hss\raise\dimen0
  \hbox{\hskip.2em$\scriptscriptstyle\circ$}\hss}}#1\else {\accent"17 #1}\fi}
  \def\cprime{$'$} \def\cprime{$'$} \def\cprime{$'$}
\providecommand{\bysame}{\leavevmode\hbox to3em{\hrulefill}\thinspace}
\providecommand{\MR}{\relax\ifhmode\unskip\space\fi MR }
\providecommand{\MRhref}[2]{%
  \href{http://www.ams.org/mathscinet-getitem?mr=#1}{#2}
}
\providecommand{\href}[2]{#2}

\end{document}